\documentclass{amsart}

\usepackage{verbatim}
\usepackage{amssymb}
\usepackage{amscd}
\usepackage{amsmath}

\newtheorem{theorem}{Theorem}[subsection]
\newtheorem{claim}{Claim}
\newtheorem{corollary}[theorem]{Corollary}
\newtheorem{lemma}[theorem]{Lemma}
\newtheorem{proposition}[theorem]{Proposition}
\numberwithin{equation}{subsection}

\newcommand{\Q}{\mathbb Q}
\newcommand{\Ht}{\widetilde{\mathbb H}}
\renewcommand{\H}{\mathbb H}
\newcommand{\G}{\widetilde{G}}
\newcommand{\Z}{\mathbb Z}
\newcommand{\C}{\mathbb C}
\newcommand{\R}{\mathbb R}
\newcommand{\F}{\mathbb F}
\newcommand{\A}{\mathbb A}
\newcommand{\bs}{\backslash}

\renewcommand{\Re}{\mathrm{Re}}
\newcommand{\B}{\mathfrak B}

\newcommand{\tk}{\widetilde k}
\newcommand{\tr}{\overline r}
\newcommand{\E}{\mathcal E}

\newcommand{\Sym}{\mathrm{Sym}}
\begin{document}
\title[$L$-functions on $GSp(4) \times GL(2)$ and their special values]{$L$-functions for holomorphic forms on $GSp(4) \times GL(2)$ and their special values}
\author{Abhishek Saha}
\date{\today}
\address{Department of Mathematics 253-37 \\ California Institute of Technology \\ Pasadena, California 91125 \\ USA.}
\email{saha@caltech.edu}

\bibliographystyle{plain}
\begin{abstract}
We provide an explicit integral representation for $L$-functions of pairs $(F,g)$ where $F$ is a holomorphic genus 2 Siegel newform and $g$ a holomorphic elliptic newform, both of squarefree levels and of equal weights. When $F,g$ have level one, this was earlier known by the work of Furusawa. The extension is not straightforward. Our methods involve precise double-coset and volume computations as well as an explicit formula for the Bessel model for $GSp(4)$ in the Steinberg case; the latter is possibly of independent interest. As an application, we prove an algebraicity result for a critical value of $L(s, F \times g)$. This is in the spirit of known results on critical values of triple product $L$-functions, also of degree 8, though there are significant differences.
\end{abstract}
\maketitle
\section*{Introduction}
$L$-functions for automorphic forms on reductive groups are objects of considerable number theoretic interest. They codify the relationship between arithmetic and analytic objects and enable us to investigate properties that are otherwise not easily accessible. One of the tools that has been successfully used to study $L$-functions and their special values is the method of integral representations; this is sometimes called the Rankin-Selberg method after Rankin and Selberg's fundamental work in this direction. Often, sharper and more explicit results are obtained when one restricts attention to holomorphic forms. The papers~\cite{garhar},~\cite{grokud},~\cite{harkud} treating the triple-product $L$-function, are good examples, and in fact, provided an inspiration for this work.

Let $\pi=\otimes \pi_v$, $\sigma = \otimes \sigma_v$ be irreducible, cuspidal automorphic representations of $GSp_4(\A)$, $GL_2(\A)$ respectively, where $\A$ denotes the ring of adeles over $\Q$. In this paper we are interested in the
degree eight $L$-function $L(s, \pi \times \sigma)$. Furusawa~\cite{fur} discovered an integral representation for this $L$-function; however, he computed the local zeta integral only in the case when $\pi_p$, $\sigma_p$ are both unramified. For several applications, this is not enough. To give an example, suppose $F= \Sym^3(E_1)$ is a holomorphic Siegel cusp form that arises as the symmetric cube of an elliptic curve $E_1$ over $\Q$ (as worked out by Ramakrishnan-Shahidi in~\cite{ramshi}) and $ g$ is an holomorphic elliptic cusp form associated to another elliptic curve $E_2$ over $\Q$. Then neither $F$ nor $ g$ can be of full level (since there exists no elliptic curve over $\Q$ that is unramified everywhere). Furthermore, the local components of the representations associated to $F$ and $ g$ at all ramified places are Steinberg.  So, if we wish to study $L(s, F \times g)$ in this case, we would need to evaluate the local zeta integral when one or both of the local representations is Steinberg.

In order to state the results of this paper, we first recall the integral representation of~\cite{fur} in detail.

Fix automorphic representations $\pi, \sigma$ as above with trivial central characters. For a quadratic extension $L/\Q$, consider the unitary group $GU(2,2) = GU(2,2;L)$. Let $P$ be the maximal parabolic of $GU(2,2)$ with a non-abelian unipotent radical. Note that $GL_1(L) \times GL_2(\Q)$ embeds naturally inside a Levi component of $P(\Q)$. So, given a Hecke character $\Lambda$ of $L$, we can use $\sigma$ and $\Lambda$ to construct an automorphic representation $\Pi$ of $P(\A)$ and thus an induced representation $I(\Pi,s) = \text{Ind}_{P(\A)}^{GU(2,2)(\A)}(\Pi \times \delta_P^s)$. In the usual manner we then define an Eisenstein series $E(g,s;f)$ on $GU(2,2)(\A)$ for an analytic section $f \in I(\Pi,s)$.

For an vector $\Phi$ in the space of $\pi$ and an analytic section $f \in I(\Pi,s)$ consider the global integral \begin{equation}\label{intrglobal}Z(s) =  \int_{Z(\A)GSp_4(\Q) \bs GSp_4(\A)}E(g,s;f)\Phi(g)dg.\end{equation}

(We refer the reader to the beginning of section~\ref{s:rankin} for our normalization of the measure $dg$.)

In~\cite{fur}, Furusawa proves the following results:\begin{enumerate}
\item For suitable choices of $L, \Lambda$ and $f$, $Z(s)$ is Eulerian, that is $$Z(s) = \prod_v Z_v(s)$$ where for each place $v$ of $\Q$, $Z_v(s)$ is an explicit local zeta integral.

\item Let $p$ be a finite prime such that \emph{$\pi_p$ and $\sigma_p$ are both unramified}. Then $$Z_p(s) = C(s) \times L(3s + \frac{1}{2},\pi_p \times \sigma_p),$$ where $C(s)$ is an explicit normalizing factor.
    \end{enumerate}

We now state the main local result of this paper. For the more precise version, see the Theorems~\ref{t:unramifiedsteinberg},
~\ref{t:steinbergsteinberg},~\ref{t:steinbergunramified}.
\\ \\
\textbf{Theorem A.} \emph{Let $p$ be a finite prime which is inert in $L$.}
 \begin{enumerate}
 \item \emph{Suppose that $\pi_p$ is unramified and $\sigma_p$ is an unramified quadratic twist of the Steinberg representation. Also suppose that $\Lambda_p$ is unramified. Then we have $$Z_p(s) = \frac{1-p^{-6s-3}}{p^2 +1} \times L(3s + \frac{1}{2},\pi_p \times \sigma_p).$$}
\item \emph{Suppose that $\pi_p$ is an unramified quadratic twist of the Steinberg representation and  $\sigma_p$ is unramified. Also suppose that $\Lambda_p$ has conductor $p$. Then we have $$Z_p(s) = \frac{1}{(p+1)(p^2 +1)} \times L(3s + \frac{1}{2},\pi_p \times \sigma_p).$$}

\item \emph{Suppose that  $\pi_p$ , $\sigma_p$ are both unramified quadratic twists of the Steinberg representations. Also suppose that $\Lambda_p$ has conductor $p$. Then we have $$Z_p(s) = \frac{p^{-6s-3}}{p(p^2 +1)(1-a_pw_pp^{-3s-\frac{3}{2}})} \times L(3s + \frac{1}{2},\pi_p \times \sigma_p),$$ where $a_p$ is the eigenvalue of the local operator $T_p$ for $\sigma_p$ and $w_p$ is the eigenvalue of the local Atkin-Lehner operator for $\pi_p$.}
 \end{enumerate}

 As already noted, the simplest case where both local representations are unramified was proved in~\cite{fur}. However the methods employed for that case are not sufficient to deal with the above three cases. The explicit evaluation of the local zeta integral involves several steps. First of all, we need to perform certain technical volume and double-coset computations. These computations --- easy in the unramified case --- are tedious and challenging for the remaining cases and are carried out in Section~\ref{s:strategypadic}. Secondly, it is necessary to suitably choose the sections of the Eisenstein series at the bad places to insure that the local zeta integrals do not vanish. Thirdly, and perhaps most crucially, the local computations require an explicit knowledge of the local Whittaker and Bessel functions. The formulae for the Whittaker model are well known in all cases; the same, however, is not true for the Bessel model. In fact, the only case where the local Bessel model for a finite place was computed before this work was when $\pi_p$ is unramified~\cite{bff,sug}. However, that does not suffice for the two cases when we have $\pi_p$ Steinberg. As a preparation for the calculations in these cases, we find, in Section~\ref{s:appendix}, \emph{an explicit formula for the Bessel function for $\pi_p$ when it is Steinberg}. This is perhaps of independent interest.

Putting together our local computations we get an integral representation for a pair $(F,g)$ of global newforms as described next.

For a square-free integer $M$ let $B(M)$ denote the congruence subgroup of $Sp(4,\Z)$ defined by $$B(M)=Sp(4,\Z) \cap  \begin{pmatrix}\Z& M\Z&\Z&\Z\\\Z& \Z&\Z&\Z\\M\Z& M\Z&\Z&\Z\\M\Z&N \Z&M\Z&\Z\\\end{pmatrix}.$$ We say that a holomorphic Siegel cusp form of genus 2 is a \emph{newform of level $M$} if:
\begin{enumerate}
\item It lies in the orthogonal complement of the space of oldforms for $B(M)$ as defined by Schmidt~\cite{sch}.
\item It is an eigenform for the Hecke algebra at all primes not dividing $M$.
\item It is an eigenform for the Atkin-Lehner operator at all primes dividing $M$.
  \end{enumerate}

For a square-free integer $N$, we call a holomorphic elliptic cusp form a newform of level $N$ if it is a newform with respect to the group $\Gamma_0(N)$ in the usual sense.

Now, fix odd, square-free positive integers $M$, $N$ and let $F$ be a genus 2 Siegel newform of level $M$ and $ g$ an elliptic newform of level $N$. We assume that $F$ and $ g$ have the same even integral weight $l$ and have trivial central characters. Furthermore we make the following (mild) assumption about $F$:

Suppose $$F(Z) = \sum_{S>0} a(S)e(\text{tr}(SZ)) $$ is the Fourier expansion; then we assume that

\begin{equation}\label{fundamentalrestriction0}a(T) \neq 0
\text{ for some } T
= \begin{pmatrix} a& \frac{b}{2} \\ \frac{b}{2} & c \end{pmatrix}\end{equation} such that $-d = b^2
-4ac$  is the discriminant of the imaginary quadratic field $\Q(\sqrt{-d}),$
  and all primes dividing $MN$ are inert in  $\Q(\sqrt{-d})$.

Let $\Phi$ denote the adelization of $F$. The representation of $GSp(4)(\A)$ generated by $\Phi$ may not be irreducible, but we know~\cite{sch}  that
all its irreducible components are isomorphic. Let us denote any of these components by $\pi$. Also, we know that
$ g$ generates an irreducible representation $\sigma$ of $GL_2(\A)$.
We prove (see Theorem~\ref{t:globalmain} for the full statement) the following result:\\ \\
\textbf{Theorem B.} \emph{Let $F, g, \pi, \sigma$ be as defined above. Then, for a suitable choice of $\Lambda, f$, the global integral defined in~\eqref{intrglobal} satisfies}

$$Z(s) = C(s) \times L(3s + \frac{1}{2},\pi \times \sigma),$$ where $C(s)$ \emph{is an explicit normalizing factor}.\\

Using the above integral representation, one can prove a certain special value result. Before stating that, we make some general remarks. If $L(s)$ is an arithmetically defined (or motivic) $L$-series, it is interesting to study its value at certain critical points $s=m$. For these critical points, the standard conjectures predict that $L(m)$ is the product of a suitable transcendental number $\Omega$ and an algebraic number $A(m)$. Moreover, it is expected that the same $\Omega$ works for $L_{\chi}(m)$ where $\chi$ is a Dirichlet character of appropriate parity.

As a consequence of Theorem B, we get, using standard algebraicity results related to Siegel modular forms and Eisenstein series~\cite{gar2,har,miz}, the following special value result. This fits into the framework of the conjectures mentioned above.\\ \\
\textbf{Theorem C.} \emph{Suppose $F, g $ are as defined above and moreover have totally real algebraic Fourier coefficients.}
\emph{Then, assuming $l >6$, we have} \begin{equation}\label{e:introl1}\frac{L(\frac{l}{2}-1, F \times g)}{\pi^{5l-8}\langle F, F\rangle \langle  g,  g \rangle}\in \overline{\Q}\end{equation}

\emph{where $\langle \ \rangle$ denotes the Petersson inner product.}\\

We should note that Theorem C was previously proved in the basic case $M=1, N=1$ by Furusawa~\cite{fur} and (independently) by Bernhard Heim~\cite{heim}, who used a different integral representation. Subsequently, B\"ocherer and Heim were able to treat the case of different weights~\cite{heimboch}. After this paper had been essentially completed, it was brought to the attention of the author that Pitale and Schmidt~\cite{pitsch} have independently, and around the same time as this paper, evaluated the local Furusawa integral $Z_p(s)$ in the case when $\pi_p$ is unramified but $\sigma_p$ is Steinberg. This allows them to prove analogues of Theorem B and Theorem C in the case $M=1, N\ge1$ square-free. They also compute the Archimedean integral for a larger family of Archimedean representations $\sigma_\infty$.

However, all the works mentioned above only consider the case when $F$ is of full level, i.e. $M=1$. This paper, to our best knowledge, is the first that gives an integral representation or proves any special value result for $L(s, F \times g)$ when $M>1$.

It is of interest to find, in addition, a reciprocity law relating to the above special value, that is, the equivariance of the action of Aut($\C$) on the quantity defined in~\eqref{e:introl1}. Unfortunately, not enough is known about the corresponding action on the Fourier coefficients of our Eisenstein series to resolve this question here. In a sequel to this paper~\cite{sah2}, we will use a certain pullback formula to get a modified integral representation for our $L$-function that involves a well-understood Siegel Eisenstein series on $GU(3,3)$. This will enable us to answer the Aut($\C$) equivariance and related questions. We say a little more about these techniques in the final section of this paper. The referee has pointed out that this way of formulating the integral representation in terms of pullbacks from $GU(3,3)$ has been anticipated by the experts.

The details of some routine calculations have been omitted from this paper for the sake of brevity; the reader who wishes to see them can take a look at the longer version of this paper available online~\cite{lflong}.

\subsection*{Notation}

\begin{itemize}

\item The symbols $\Z$, $\Z_{\ge0}$, $\Q$, $\R$, $\C$, $\Z_p$ and $\Q_p$ have the usual meanings. $\A$ denotes the ring of adeles of $\Q$. For a complex number $z$, $e(z)$ denotes $e^{2\pi i z}$.\\

\item For any commutative ring $R$ and positive integer $n$,
$M_n(R)$ denotes the ring of $n$ by $n$ matrices with entries in $R$ and $GL_n(R)$ denotes the group of
invertible matrices in $M_n(R)$. If $A\in M_n(R)$, we let $A^T$ denote its transpose. We use $R^{\times}$ to denote $GL_1(R)$.\\

\item Denote by $J_n$ the $2n$ by $2n$ matrix given by $$J_n =
\begin{pmatrix}
0 & I_n\\
-I_n & 0\\
\end{pmatrix}.$$\\ We use $J$ to denote $J_2.\\$

\item For a positive integer $n$ define the group $GSp(2n)$ by $$GSp(2n,R) = \{g \in GL_{2n}(R) | g^TJ_ng = \mu_n(g)J_n,
\mu_n(g) \in R^\times\}$$ for any commutative ring $R$.

Define $Sp(2n)$ to be the subgroup of $GSp(2n)$ consisting of elements $g_1\in GSp(2n)$ with $\mu_n(g_1)=1$.

The letter $G$ will always stand for the group $GSp(4)$ and $G_1$ for the group $Sp(4)$.\\

\item For a commutative ring $R$ we denote by $I(2n, R)$ the Borel subgroup of $GSp(2n,R)$ consisting of the set of matrices that look like $\begin{pmatrix}
A & B\\
0 & \lambda (A^T)^{-1}\\
\end{pmatrix}.$ where $A$ is lower-triangular and $\lambda \in R^\times$.
Denote by $B$ the Borel subgroup of $G$ defined by $B = I(4)$
and $U$ the subgroup of $G$ consisting of matrices that look like $\begin{pmatrix}
\ast & 0 & \ast & \ast\\
\ast & \ast & \ast & \ast\\
\ast & 0 & \ast & \ast\\
0 & 0 & 0  & \ast\\
\end{pmatrix}.$\\

\item
For a quadratic extension $L$ of $\Q$ define $$ GU(n,n)= GU(n,n;L)$$ by
$$GU(n,n)(\Q) =  \{g \in GL_{2n}(L) | (\overline{g})^TJ_ng = \mu_n(g)J_n, \mu_n(g)
\in \Q^\times\}$$ where $\overline{g}$ denotes the conjugate of $g$.

Denote the algebraic group $GU(2,2;L)$ by $\widetilde{G}$.\\

\item Define \begin{align*}\Ht_n&= \{ Z \in M_{2n}(\C) | i( \overline{Z} - Z) \text{ is positive definite}\},\\
    \H_n&= \{ Z \in M_n(\C) | Z =Z^T,i( \overline{Z} - Z) \text{ is positive definite}\}.\end{align*}
For $g=
\begin{pmatrix} A&B\\ C&D \end{pmatrix} \in \G(\R)$, $Z\in \Ht_2$ define $$J(g,Z) = CZ + D.$$ The same definition works for $g\in G(\R), Z\in \H_2.$\\

\item For $v$ be a finite place of $\Q$, define $L_v =L\otimes_\Q \Q_v$.

$\Z_L$ denotes the ring of integers of $L$ and $\Z_{L,v}$ its $v$-closure in $L_v$.

Define maximal compact subgroups $\widetilde{K_v}$ and $K_v$ of $\G(\Q_v)$ and $G(\Q_v)$ respectively by
\begin{align*}
\widetilde{K_v}&=\G(\Q_v) \cap GL_4(\Z_{L,v}),\\K_v &= G(\Q_v) \cap GL_4(\Z_v).\end{align*}

\item For a positive integer $N$ the subgroups $\Gamma_0(N)$ and $\Gamma^0(N)$ of $SL_2(\Z)$ are defined by $$ \Gamma_0(N) =\{ A \in SL_2(\Z) \mid A \equiv \begin{pmatrix}
\ast & \ast\\
0 & \ast
\end{pmatrix}\pmod{N} \},$$
$$ \Gamma^0(N) =\{ A \in SL_2(\Z) \mid A \equiv \begin{pmatrix}
\ast & 0\\
\ast & \ast
\end{pmatrix}\pmod{N} \}.$$
For $p$ a finite place of $\Q$, their local analogues $\Gamma_{0,p}$ (resp. $\Gamma^0_p$) are defined by $$
\Gamma_{0,p}= \{ A \in GL_2(\Z_p) \mid A \equiv \begin{pmatrix}
\ast & \ast\\
0 & \ast
\end{pmatrix}\pmod{p} \},$$
$$ \Gamma^0_p =\{ A \in GL_2(\Z_p) \mid A \equiv \begin{pmatrix}
\ast & 0\\
\ast & \ast
\end{pmatrix}\pmod{p} \}.$$

The local \emph{Iwahori} subgroup $I_p$
 is defined to be the subgroup of $K_p= G(\Z_p)$ consisting of
those elements of $K_p$ that when reduced mod $p$ lie in the Borel subgroup of $G(\mathbb{F}_p)$. Precisely,
$$I_p = \{ A \in K_p \mid A \equiv \begin{pmatrix}
\ast & 0 & \ast & \ast\\
\ast & \ast & \ast & \ast\\
0 & 0 & \ast & \ast\\
0 & 0 & 0  & \ast\\
\end{pmatrix}\pmod{p} \}$$

\end{itemize}

\section{Preliminaries}\label{prelim}

\subsection{Bessel models}\label{s:bessel}

We recall the definition of the Bessel model of Novodvorsky and Piatetski-Shapiro \cite{nov} following the
exposition of Furusawa
\cite{fur}.\\

Let $S \in M_2(\Q)$ be a symmetric matrix. We let disc($S$)
 $=-4 $ det$(S)$ and put $d=-$disc$(S)$. If $S
=\begin{pmatrix}
a & b/2\\
b/2 & c\\
\end{pmatrix}$ then we define the element $\xi = \xi_S = \begin{pmatrix}
b/2 & c\\
-a & -b/2\\
\end{pmatrix}$.

Let $L$ denote the subfield $\Q(\sqrt{-d})$ of $\C$.

We always identify $\Q(\xi)$ with $L$ via
\begin{equation}\label{e:L} \Q(\xi)\ni x + y\xi \mapsto x +
y\frac{\sqrt{-d}}{2} \in L, x,y\in \Q. \end{equation}

We define a subgroup $T =T_S$ of $GL_2$ by \begin{equation}T(\Q) = \{g \in GL_2(\Q) | g^TSg =det(g)
S\}.\end{equation} It is not hard to verify that $T(\Q) = \Q(\xi)^\times$. We identify $T(\Q)$ with
$L^\times$ via~\eqref{e:L}).\\

We can consider $T$ as a subgroup of $G$ via \begin{equation}T \ni g \mapsto
\begin{pmatrix}
g & 0\\
0 & det(g).(g^{-1})^T\\
\end{pmatrix} \in G.\end{equation}

Let us denote by $U$ the subgroup of $G$ defined by $$U = \{u(X) =
\begin{pmatrix}
1_2 & X\\
0 & 1_2\\
\end{pmatrix} |{}X^T = X\}.$$
Let $R$ be the subgroup of $G$ defined by $R=TU$.\\

Let $\psi$ be a non trivial character of $\A / \Q$. We define the character $\theta = \theta_S$ on $U(\A)$ by
$\theta(u(X))= \psi(tr(S(X)))$. Let $\Lambda$ be a character of $T(\A) / T(\Q)$ such that $\Lambda | \A^\times= 1$. Via~\eqref{e:L} we can think of $\Lambda$ as a character of $L^\times(\A)/L^\times$ such
that $\Lambda | \A^\times = 1$. Denote by $\Lambda \otimes \theta$ the character of $R(\A)$ defined by $(\Lambda
\otimes \theta)(tu) = \Lambda(t)\theta(u)$ for $t\in
T(\A)$ and $u\in U(\A).$ \\

Let $\pi$ be an automorphic cuspidal representation of $G(\A)$ with trivial central character and $V_\pi$ be its
space of automorphic forms.

Then for $\Phi \in V_\pi$, we define a function $B_\Phi$ on $G(\A)$ by \begin{equation}B_\Phi(h) =
\int_{R(\A)/R(\Q)Z_G(\A)} (\Lambda \otimes \theta)(r)^{-1}\Phi(rh)dr.\end{equation}

The $\C$ - vector space of function on $G(\A)$ spanned by $\{B_\Phi | \Phi \in V_\pi \}$ is called the global
Bessel space of type $(S, \Lambda, \psi)$ for $\pi$. We say that $\pi$ has a global Bessel model of type $(S,
\Lambda, \psi),$ if the global Bessel space has positive dimension, that is if there exists $\Phi \in V_\pi$
such that $B_\Phi \neq 0$. In Sections 1--7 of this paper, we assume that:
\begin{equation} \label{e:assump}  \text{ \emph{There exists} } S, \Lambda, \psi \text{ \emph{such that} } \pi
\text{ \emph{has a global Bessel model of type} } (S, \Lambda, \psi).
\end{equation}

\subsection{Eisenstein series}\label{s:eisenstein}

We briefly recall the definition of the Eisenstein series used by Furusawa in \cite{fur}. Let $P$ be the maximal
parabolic subgroup of $\widetilde{G}$ consisting of the elements in $\G$ that look like$
\begin{pmatrix}
\ast & \ast & \ast & \ast\\
0 & \ast & \ast & \ast\\
0 & \ast & \ast & \ast\\
0 & \ast & \ast & \ast\\
\end{pmatrix}$. We have the Levi
decomposition $P= MN$ with  $M= M^{(1)}M^{(2)}$ where the groups $M, N, M^{(1)}, M^{(2)}$ are as defined in
\cite{fur}.

Precisely,
\begin{equation}
M^{(1)}(\Q) =\left\{\begin{pmatrix}
a & 0 & 0 & 0\\
0 & 1 & 0 & 0\\
0 & 0 & \overline{a}^{-1} & 0\\
0 & 0 &0  & 1\\
\end{pmatrix} \mid a \in L^\times \right \} \simeq L^\times.
\end{equation}

\begin{equation}\begin{split}
M^{(2)}(\Q) &=\left\{\begin{pmatrix}
1 & 0 & 0 & 0\\
0 & \alpha & 0 & \beta\\
0 & 0 & \lambda & 0\\
0 & \gamma
 &0  & \delta\\
\end{pmatrix} | \begin{pmatrix} \alpha & \beta\\ \gamma & \delta \end{pmatrix} \in GU(1,1)(\Q), \lambda =
\mu_1 \begin{pmatrix} \alpha & \beta\\ \gamma & \delta
\end{pmatrix}\right \} \\ &\simeq GU(1,1)(\Q).\end{split}
\end{equation}

\begin{equation}
N(\Q) =\left\{\begin{pmatrix}
1 & x & 0 & 0\\
0 & 1 & 0 & 0\\
0 & 0 & 1 & 0\\
0 & 0 & -\overline{x}  & 1\\
\end{pmatrix}\begin{pmatrix}
1 & 0 & a & y\\
0 & 1 & \overline{y} & 0\\
0 & 0 & 1 & 0\\
0 & 0 & 0 & 1\\
\end{pmatrix} \mid a \in \Q , x,y \in L \right \} .
\end{equation}

We also write $$m^{(1)}(a) = \begin{pmatrix}
a & 0 & 0 & 0\\
0 & 1 & 0 & 0\\
0 & 0 & \overline{a}^{-1} & 0\\
0 & 0 &0  & 1\\
\end{pmatrix}, \qquad m^{(2)}\begin{pmatrix} \alpha & \beta\\ \gamma & \delta \end{pmatrix}= \begin{pmatrix}
1 & 0 & 0 & 0\\
0 & \alpha & 0 & \beta\\
0 & 0 & \lambda & 0\\
0 & \gamma &0  & \delta\\
\end{pmatrix}.$$
Let $\sigma$ be an irreducible automorphic cuspidal representation of $GL_2(\A)$ with central character
$\omega_\sigma$. Let $\chi_0$ be a character of $L^\times(\A)/L^\times$ such that $\chi_0 \mid \A^\times =
\omega_\sigma.$

Finally, let $\chi$ be a character of $L^\times(\A)/ L^\times = M_1(\A)M_1(\Q)$ defined by
\begin{equation} \chi(a) =
\Lambda(\overline{a})^{-1}\chi_0(\overline{a})^{-1}. \end{equation}

Then defining \begin{equation}\Pi(m_1m_2)= \chi(m_1)(\chi_0 \otimes \sigma)(m_2) , m_1 \in M_1(\A), m_2 \in
M_2(\A)\end{equation} we extend $\sigma$ to an automorphic representation $\Pi$ of $M(\A)$. We regard $\Pi$ as a
representation of $P(\A)$ by extending it trivially on $N(\A)$. Let $\delta_P$ denote the modulus character of
$P$. If $p=m_1m_2n \in P(\A)$ with $m_i \in M_i(\A) (i = 1,2)$ and $n \in N(\A)$,
\begin{equation} \delta_P(p) = | N_{L/\Q}(m_1)|^3 \cdot |
\mu_1(m_2) |^{-3}, \end{equation} where $\mid\mid$ denoted the modulus function on $\A$.

Then for $s\in \C$, we form the family of induced automorphic representations of $\widetilde{G}(\A)$
\begin{equation}I(\Pi, s) = \text{Ind}_{P(\A)}^{\widetilde{G}(\A)}(\Pi \otimes \delta_P^s)\end{equation} where
the induction is normalized. Let $f(g,s)$ be an entire section in $I(\Pi,s)$ viewed concretely as a
complex-valued function on $\G(\A)$ which is left $N(\A)$-invariant and such that for each fixed $g\in \G(\A)$,
the function $m \mapsto f(mg,s)$ is a cusp form on $M(\A)$ for the automorphic representation $\Pi \otimes
\delta_P^s$. Finally we form the Eisenstein series $E(g,s) = E(g,s;f)$ by \begin{equation}E(g,s) = \sum_{\gamma
\in P(\Q)\backslash \G(\Q)} f(\gamma g,s)\end{equation} for $g\in \G(\A)$.

This series converges absolutely (and uniformly in compact subsets) for $\Re(s)> 1/2$, has a meromorphic
extension to the entire plane and satisfies a functional equation (see \cite{langlands,fur} ).

\section{The Rankin-Selberg integral}\label{s:rankin}
\subsection{The global integral}
Before defining our fundamental global integral, let us clarify our normalization on $G(\Q)Z_G(\A)\backslash G(\A)$. We may factor $G(\A)$ as a restricted direct product $\prod'G_v$ where $v$ runs over all places of $\Q$. Let $K_v = G(\Z_v)$ if $v$ is a finite place. Let $K_\infty = \mathrm{Sp}(4,\R) \cap O(4, \R)$. Thus at each place $K_v$ is a maximal compact subgroup of $G_v$. We define a Haar measure $dg$ on $Z_G(\A)\backslash G(\A)$ as the product $\prod_v dg_v$ where $dg_v$ is the Haar measure on $Z_{G_v}\backslash G_v$ normalized such that $(Z_{G_v} \cap K_v)\backslash K_v$ has volume 1. Thus $dg$ leads to a natural quotient measure on $G(\Q)Z_G(\A)\backslash G(\A)$ which we also denote by $dg$.

We note that there is a canonical isomorphism $$G(\Q)Z_G(\A)\backslash G(\A)/ \prod_v K_v = \mathrm{Sp}_4(\Z) \bs \H_2.$$
A $\mathrm{Sp}_4(\R)$ invariant measure on $\H_2$ is given by $dh = (\det Y)^{-3}dXdY$. It can be easily checked that under the above isomorphism, we have $dg = \frac{1}{2}dh.$ 

The main object of study in this paper is the following global integral of Rankin-Selberg type \begin{equation}
Z(s)=Z(s,f,\Phi)=\int_{G(\Q)Z_G(\A)\backslash G(\A)}E(g,s,f)\Phi(g)dg,
\end{equation} where $\Phi \in V_\pi$ and $f \in I(\Pi, s).$ $Z(s)$
converges absolutely away from the poles of
the Eisenstein series.\\

Let $\Theta = \Theta_S$ be the following element of $\G(\Q)$ $$\Theta =\begin{pmatrix}
1 & 0 & 0 & 0\\
\alpha & 1 & 0 & 0\\
0 & 0& 1 & -\overline{\alpha}\\
0 & 0 & 0 & 1\\
\end{pmatrix} \text{ where } \alpha = \frac{b + \sqrt{-d}}{2c}$$

The `basic identity' proved by Furusawa in \cite{fur} is that
\begin{equation}
Z(s) =\int_{R(\A)\backslash G(\A)}W_f(\Theta h,s)B_\Phi(h)dh
\end{equation}
where for $g\in \G(\A)$ we have \begin{equation}W_f(g,s) = \int_{\A/\Q}f \left(\begin{pmatrix}
1 & 0 & 0 & 0\\
0 & 1 & 0 & x\\
0 & 0& 1 & 0\\
0 & 0 & 0 & 1\\
\end{pmatrix}g,s \right)\psi(cx)dx. \end{equation} and $B_\Phi$ is
the Bessel model of type $(S, \Lambda, \psi)$ defined in section~\ref{prelim}.
\subsection{The local integral}
In this section $v$ refers to any place of $\Q$. Let $\pi=\otimes_v \pi_v$ and $\sigma=\otimes_v \sigma_v$. Now
suppose that $\Phi$ and $f$ are factorizable functions with $\Phi = \otimes_v \Phi_v$ and $f( \ ,s)= \otimes_v
f_v( \ ,s).$

By the uniqueness of the Whittaker and the Bessel models, we have
\begin{equation}
W_f(g,s) = \prod_v W_{f,v}(g_v,s)
\end{equation}
\begin{equation}
B_\Phi(h)= \prod_v B_{\Phi,v}(h_v,s)
\end{equation}
for $g =(g_v)\in \G(\A)$ and $h =(h_v) \in G(\A)$ and \emph{local} Whittaker and Bessel functions $W_{f,v}$ ,
$B_{\Phi,v}$ respectively. Henceforth we write $W_v = W_{f,v} , B_v = B_{\Phi,v}$ when no confusion can arise.

Therefore our global integral breaks up as a product of local integrals \begin{equation}\label{e:local} Z(s)
=\prod_v Z_v(
s)
\end{equation} where
$$
Z_v(s) = Z_v(s,W_{v},B_{v}) = \int_{R(\Q_v)\backslash G(\Q_v)}W_{v}(\Theta g,s)B_{v}(g)dg.
$$

\subsection{The unramified case}
The local integral is evaluated in \cite{fur} in the unramified case. We recall the result here.

Suppose that the characters $\omega_\pi , \omega_\sigma, \chi_0$ are trivial. Now let $q$ be a finite prime of
$\Q$ such that
\begin{enumerate}

\item The local components  $\pi_q , \sigma_q$ and $\Lambda_q $ are all unramified.

\item The conductor of  $\psi_q$ is
 $\Z_q.$

\item $S = \begin{pmatrix} a & b/2 \\ b/2 & c\\ \end{pmatrix} \in
M_2(\Z_q)$  with  $c\in \Z_q^\times.$

\item $-d = b^2 - 4ac$ generates the discriminant of $L_q/ \Q_q.$

\end{enumerate}

Since $\sigma_q$ is spherical, it is the spherical principal series representation induced from unramified
characters $\alpha_q, \beta_q$ of $\Q_q ^\times$.

Suppose $M_0$ is the maximal torus (the group of diagonal matrices) inside $G$ and $P_0$ the Borel subgroup containing $M_0$ as Levi component. $\pi_q$ is a spherical principal series representation, so there exists an unramified character $\gamma_q$ of $M_0(\Q_q)$ such that $\pi_q=Ind_{P_0(\Q_q)}^{M_0(\Q_q)}\gamma_q$ , (where we extend $\gamma_q$ to $P_0$ trivially). We define characters $\gamma_q^{(i)} (i=1,2,3,4)$ of $\Q_q^\times$ by \begin{align*}\gamma_q^{(1)}(x)&=\gamma_q\begin{pmatrix}x&0&0&0\\0&x&0&0\\0&0&1&0\\0&0&0&1\end{pmatrix}, & \gamma_q^{(2)}(x)&=\gamma_q\begin{pmatrix}x&0&0&0\\0&1&0&0\\0&0&1&0\\0&0&0&x\end{pmatrix},\\
\gamma_q^{(3)}(x)&=\gamma_q\begin{pmatrix}1&0&0&0\\0&1&0&0\\0&0&x&0\\0&0&0&x\end{pmatrix},&
\gamma_q^{(4)}(x)&=\gamma_q\begin{pmatrix}1&0&0&0\\0&x&0&0\\0&0&x&0\\0&0&0&1\end{pmatrix}.\end{align*}

Now let $f_q(\ ,s)$ be the unique normalized $\widetilde{K_q}- $ spherical vector in $I_q(\Pi_q,s)$ and $\Phi_q$
be the unique normalized $K_q- $ spherical vector in $\pi_q$. Let $W_q, B_q$ be the coresponding vectors in the
local Whittaker and Bessel spaces. The following result is proved in \cite{fur}.

\begin{theorem}[Furusawa]\label{t:furusawa}
Let $\rho(\Lambda_q)$ denote the Weil representation of $GL_2(\Q_q)$ corresponding to $\Lambda_q$. Then we have
$$Z_q(s,W_q, B_q) = \frac{L(3s + \frac{1}{2} , \pi_q
\times \sigma_q)}{L(6s+1, \mathbf{1})L(3s + 1, \sigma_q \times \rho(\Lambda_q))}$$

where,
 \begin{align*}L(s , \pi_q \times
\sigma_q)&=\prod_{i=1}^4 \left( (1-\gamma_q^{(i)}\alpha_q(q) q^{-s})(1-\beta_q^{(i)}\alpha_q(q) q^{-s})\right)^{-1},\\
L(s, \mathbf{1}) &= (1 - q^{-s})^{-1},\end{align*}

$L(s, \sigma_q \times \rho(\Lambda_q)) \\= \begin{cases}
(1-\alpha_q^2(q)q^{-2s})^{-1}(1-\beta_q^2(q)q^{-2s})^{-1} &\text{if $q$ is inert in $L$, }\\
\\(1-\alpha_q(q)\Lambda_q(q_1) q^{-s})^{-1}(1-\beta_q(q)\Lambda_q(q_1) q^{-s})^{-1} & \text{if $q$ is ramified
in $L$,}\\ \\(1-\alpha_q(q)\Lambda_q(q_1) q^{-s})^{-1}(1-\beta_q(q)\Lambda_q(q_1)
q^{-s})^{-1}\\
\cdot(1-\alpha_q(q)\Lambda_q^{-1}(q_1) q^{-s})^{-1}(1-\beta_q(q)\Lambda_q^{-1}(q_1) q^{-s})^{-1}& \text{if $q$
splits in $L$,}
\end{cases}
$\\

where $q_1 \in \Z_q \otimes_\Q L$ is any element with $N_{L/\Q}(q_1) \in q\Z_q^\times$.
\end{theorem}

\section{Strategy for computing the $p$-adic integral}\label{s:strategypadic}
\subsection{Assumptions}Throughout this section we fix an odd prime $p$ in $\Q$
such that \emph{$p$ is inert in $L$}. Moreover, we assume  that $S\in M_2(\Z_p)$.

The fact that $p$ is inert in $L$ implies that if $w,z$ are elements of $\Z_p$ then $w +z\xi \in (T(\Q_p)\cap
K_p)$ if and only if at least one of $w,z$ is an unit.

Moreover the additional assumption $S\in M_2(\Z_p)$ forces that $a,c$ are units in $\Z_p$.
\subsection{An explicit set of coset representatives}
Recall the Iwahori subgroup $I_p$. It will be useful to describe a set of coset representatives of $K_p/I_p$.

But first some definitions.

Let $Y$ be the set $\{0,1,..,p-1\}$. Let $ V = Y \cup \{\infty\}$ where $\infty$ is just a convenient formal
symbol.

For $x = (n,q,r) \in \Z_p^3$, let $U_x \in U(\Q_p)$ be the matrix $\begin{pmatrix}
1 & 0 & n & q\\
0 & 1 & q & r\\
0 & 0 & 1 & 0\\
0 & 0 & 0 & 1\\
\end{pmatrix}$

For $y \in \Z_p$ define $Z_y = \begin{pmatrix}
1 & y & 0 & 0\\
0 & 1 & 0 & 0\\
0 & 0 & 1 & 0\\
0 & 0 & -y & 1\\
\end{pmatrix} \in K_p$.

Also define $Z_{\infty} =\begin{pmatrix}
0 & 1 & 0 & 0\\
1 & 0 & 0 & 0\\
0 & 0 & 0 & 1\\
0 & 0 & 1 & 0\\
\end{pmatrix} \in K_p$.\\

In particular, the definitions $U_x$, $Z_y$ make sense for $x\in Y^3, y\in V$. Now we define the following three
classes of matrices. We call them matrices of class A, class B and class D respectively.
\begin{enumerate}
\item For $x=(n,q,r) \in Y^3, y\in V$, let $A_x^y =U_xJZ_y$.
\item For $x=(n,q,r) \in Y^3 \text{ with } q^2 - nr \equiv 0 \pmod{p} \text{ and } y\in V$, let $B_x^y =JU_xJZ_y$.
\item For $\lambda,y \in V$, let $D_\lambda^y = \begin{cases} \  \begin{pmatrix}
-\lambda & 0 & 0 & 1\\
1 & 0 & 0 & \lambda^{-1}\\
0 & 1 & \lambda^{-1} & 0\\
0 & \lambda & -1 & 0\\
\end{pmatrix}Z_y & \text{  if } \lambda \neq 0, \infty,\\ \\
\ \begin{pmatrix}
0 & 0 & 0 & 1\\
-1 & 0 & 0 & 0\\
0 & 1 & 0 & 0\\
0 & 0 & 1 & 0\\
\end{pmatrix}Z_y
& \text{  if } \lambda = 0,\\ \\
\ \begin{pmatrix}
-1 & 0 & 0 & 0\\
0 & 0 & 0 & 1\\
0 & 0 & 1 & 0\\
0 & 1 & 0 & 0\\
\end{pmatrix}Z_y &
\text{  if } \lambda = \infty.
\end{cases}$\\
\end{enumerate}

Let $S$ be the set obtained by taking the union of the class A, class B and class D matrices, precisely
$S = \{A_x^y\}_{\substack{y\in V
\\x \in Y^3}} \ \bigcup \ \{B_x^y\}_{\substack{y\in V, x=(n,q,r) \in
Y^3\\q^2-nr \equiv 0 \pmod{p}}} \ \bigcup \ \{D_\lambda^y\}_{\substack{\lambda\in V\\y \in V}}$. Clearly $S$ has
cardinality $p^3(p+1) + p^2(p+1)+(p+1)^2 = (p+1)^2(p^2+1)$.
\begin{lemma}\label{l:cosetdeco}
$S$ is a complete set of coset representatives for $K_p/I_p$.
\end{lemma}
\begin{proof}
Let us first verify that $S$ has the \textsl{right} cardinality. Clearly the cardinality of $K_p/I_p$ is the
same as the cardinality of $G(\mathbb{F}_p)/B(\mathbb{F}_p)$ where $B$ is the Borel subgroup of $G$. By
\cite[Theorem 3.2]{kim}, $|G(\mathbb{F}_p)| = p^4(p-1)^3(p+1)^2(p^2+1).$ On the other hand $B$ has the
Levi-decomposition $$B=\begin{pmatrix}
g & 0\\
0 & v.(g^{-1})^T\\
\end{pmatrix}\begin{pmatrix}
1_2 & X\\
0 & 1_2\\
\end{pmatrix}$$ with $g$ upper-triangular, $X$ symmetric and $v\in GL(1)$. So $|B(\mathbb{F}_p)|=
p^4(p-1)^3$. Thus $|G(\mathbb{F}_p)/B(\mathbb{F}_p)| = (p+1)^2(p^2+1)$ which is the same as the cardinality of
$S$.

So it is enough to show that no two matrices in $S$ lie in the same coset.

For a $2 \times 2$ matrix $H$ with coefficients in $\Z_p$, we may reduce $H \bmod p$ and consider the
$\F_p$-rank of the resulting matrix; we denote this quantity by $r_p(H).$ It is easy to see that
if the matrix $A=\begin{pmatrix} A_1 & A_2 \\ A_3 & A_4\\
\end{pmatrix}$ varies in a fixed coset of $K_p/I_p$, the pair
$(r_p(A_1), r_p(A_3))$ remains constant.

Observe now that if $A$ is of class A, then $r_p(A_3) = 2$ ; for $A$ of class B, $r_p(A_3) < 2$ and $r_p(A_1) =
2$; while for $A$ of class D we have $r_p(A_3) < 2$, $r_p(A_1)< 2$. This proves that elements of $S$ of
different classes cannot lie in the same coset.

Now we consider distinct elements of $S$ of the same class, and show that they too must lie in different cosets.

For $x_1=(n_1,q_1,r_1), x_2=(n_2,q_2,r_2) \in Y^3 $, $y_1, y_2\in Y$, consider the elements $A_{x_1}^{y_1},
A_{x_2}^{y_2},B_{x_1}^{y_1},B_{x_2}^{y_2}$ of $S.$ We have

$ (A_{x_1}^{y_1})^{-1}A_{x_2}^{y_2} =(B_{x_1}^{y_1})^{-1}B_{x_2}^{y_2}=$

 $$ \begin{pmatrix} 1&y_2-y_1&0&0\\0&1&0&0\\-n_2+n_1&-n_2 y_2-q_2+n_1 y_2+q_1&1&0\\
y_1(n_1-n_2)+q_1-q_2&y_1y_2(n_1-n_2)+(y_1+y_2)(q_1-q_2)-r_2+r_1&y_1-y_2&1\\ \end{pmatrix}.$$

So if the above matrix belongs to $I_p$, we must have $y_1=y_2, n_1=n_2$. That leads to $q_1=q_2$, and finally
by looking at the bottom row we conclude $r_1=r_2.$

This covers the case of class A and class B matrices in $S$ whose $y$-component is not equal to $\infty$.

Now $ (A_{x_1}^{y_1})^{-1}A_{x_2}^{\infty} =(B_{x_1}^{y_1})^{-1}B_{x_2}^{\infty}=$

$$\begin{pmatrix}-y_1&1&0&0\\1&0&0&0\\q_1-q_2&n_1-n_2&0&1\\q_1 y_1+r_1-y_1 q_2-r_2&n_1 y_1+q_1-y_1
n_2-q_2&1&y_1\end{pmatrix}$$

which cannot belong to $I_p.$

Also $ (A_{x_1}^{\infty})^{-1}A_{x_2}^{\infty} =(B_{x_1}^{\infty})^{-1}B_{x_2}^{\infty}=$

$$\begin{pmatrix}1&0&0&0\\0&1&0&0\\-r_2+r_1&-q_2+q_1&1&0\\-q_2+q_1&-n_2+n_1&0&1\end{pmatrix}$$

and if the above matrix lies in $I_p$ we must have $x_1 = x_2$.

Thus we have completed the proof for class A and class B matrices. To complete the proof of the lemma we need to show that no two class D matrices are in the same coset. The calculations for that case are similar to those above and are therefore omitted; the reader can find them in the longer version of this paper available online~\cite{lflong}.
\end{proof}

\subsection{Reducing the integral to a sum}\label{s:intsum}
By \cite[p. 201]{fur}) we have the following disjoint union
\begin{equation}\label{e:deco}
G(\Q_p) = \coprod_{\substack{l \in \Z\\ 0 \leq m\in \Z}} R(\Q_p) \cdot h(l,m) \cdot K_p
\end{equation}
 where $$h(l,m) = \begin{pmatrix}
p^{2m+l} & 0 & 0 & 0\\
0 & p^{m+l} & 0 & 0\\
0 & 0 & 1 & 0\\
0 & 0 & 0 & p^m\\
\end{pmatrix}.$$

We wish to compute \begin{equation}\label{e:zeta}Z_p(s) = \int_{R(\Q_p)\backslash G(\Q_p)}W_p(\Theta
h,s)B_p(h)dh.
\end{equation}

By \eqref{e:deco} and \eqref{e:zeta} we have
\begin{equation}\label{e:sumdeco}
Z_p(s) = \sum_{l\in \Z, m\geq 0}\int_{R(\Q_p)\backslash R(\Q_p)h(l,m)K_p}W_p(\Theta h,s)B_p(h)dh.
\end{equation}

For $m \geq 0$ we define the subset $T_m$ of $S$ by \\

$T_m = \{B_{(1,0,0)}^0 , B_{(1,0,0)}^{\infty}, B_{(0,0,1)}^0,B_{(0,0,1)}^\infty,B_{(0,0,0)}^0
,B_{(0,0,0)}^\infty,
A_{(0,0,0)}^0,A_{(0,0,0)}^\infty \}$  \\ if $m>0$, \\

$T_0 = \{B_{(1,0,0)}^0, B_{(1,0,0)}^\infty, B_{(0,0,0)}^0, A_{(0,0,0)}^0 \}$.
\\

Also, we use the notation \label{p:t}$t_1=B_{(1,0,0)}^0 , t_2 =B_{(1,0,0)}^{\infty},...,t_8=A_{(0,0,0)}^\infty$.
Thus $T_m = \{t_i | 1 \leq i \leq 8\}$ if $m > 0$ and $T_0 = \{t_1, t_2,t_5,t_7\}.$

\begin{proposition}\label{p:cosetdeco}
Let $l \in \Z , m \geq 0$. Then we have $$R(\Q_p)\backslash R(\Q_p)h(l,m)K_p = \coprod_{t \in T_m}
R(\Q_p)\backslash R(\Q_p)h(l,m)tI_p.$$
\end{proposition}
\begin{proof}
Define two elements $f$ and $g$ in $K_p$ to be $(l,m)$-equivalent if there exists $r \in R(\Q_p)$ and $k\in I_p$
such that $rh(l,m)fk =h(l,m)g$. Furthermore observe that if two elements of $K_p$ are congruent$\mod p$ then
they are in the same $I_p$-coset and therefore are trivially $(l,m)$-equivalent.

The proposition can be restated as saying that any $s \in S$ is $(l,m)$-equivalent to exactly one of the
elements $t$ with $t\in T_m$. This will follow from the following nine claims.

\begin{claim}\label{c:cl1}Any class A matrix in $S$ by left-multiplying by an appropriate element
of $U(\Z_p)$ can be made congruent $\mod p$ to  $A_{(0,0,0)}^y$ for some $y\in V$.
\end{claim}

\begin{claim}\label{c:cl2}
If $m > 0$ all the $A_{(0,0,0)}^y, y\in V\setminus\{0\}$ are $(l,m)$-equivalent. In the case $m=0$ all the
$A_{(0,0,0)}^y, y\in V$ are $(l,0)$-equivalent.
\end{claim}

\begin{claim}\label{c:bclaim}Any class B matrix in $S$ by left-multiplying by an appropriate element
of $U(\Z_p)$ can be made congruent $\mod p$ to one of the matrices
$$B_{(1,\lambda,\lambda^2)}^{-\lambda},B_{(1,\lambda,\lambda^2)}^{\infty}
B_{(0,0,1)}^y, B_{(0,0,0)}^y,$$ where $\lambda \in Y,y\in V$.
\end{claim}

\begin{claim} \label{c:cl4}The matrices $B_{(1,\lambda,\lambda^2)}^y, \lambda \in Y, y\in \{-\lambda,\infty\}$ are all
$(l,m)$-equivalent to one of the matrices $B_{(1,0,0)}^y , y\in \{0,\infty\}$.
\end{claim}

\begin{claim} \label{c:cl5}The matrices $B_{(0,0,1)}^y, y\in V$ by left-multiplying by an appropriate element of $U(\Z_p)$
can be made equal to one of the matrices $B_{(0,0,1)}^y$ with $y\in \{0,\infty\}$.
\end{claim}

\begin{claim} \label{c:cl6}The matrices $B_{(0,0,0)}^y, y\in V$ are $(l,m)$-equivalent to one of the matrices $B_{(0,0,0)}^y$ with $y\in\{0,\infty\}$.
In the case $m=0$ these two matrices are also equivalent.
\end{claim}

\begin{claim} \label{c:cl7}The matrices $B_{(1,0,0)}^0, B_{(0,0,1)}^\infty$ are
$(l,0)$-equivalent and the matrices \\ $B_{(1,0,0)}^\infty, B_{(0,0,1)}^0$ are also $(l,0)$-equivalent.
\end{claim}
\begin{claim}\label{c:cl8}
Any class D matrix $D_\lambda^{y}$ by left-multiplying by an appropriate element of $U(\Z_p)$ can be made equal
to a class B matrix.
\end{claim}
\begin{claim}\label{c:cl9}
No two elements of $T_m$ are $(l,m)$-equivalent for any $m \geq 0$.
\end{claim}

Indeed Claims~\ref{c:cl1}, \ref{c:cl2} imply that any class A matrix is $(l,m)$-equivalent to one of $t_7, t_8$
(and when $m=0$,  $t_7$ alone suffices). On the other hand
claims~\ref{c:bclaim},\ref{c:cl4},\ref{c:cl5},\ref{c:cl6},\ref{c:cl7} tell us that any class B matrix is
$(l,m)$-equivalent to one of the $t_i, 1\leq i\leq 6$ (and that just $t_1, t_2, t_5$ suffice if $m=0$). Also
claim~\ref{c:cl8} says that any class D matrix is also $(l,m)$-equivalent to one of the above. Since the class
A, class B and class D matrix exhaust $S$, this shows that any element of $S$ is $(l,m)$-equivalent to some
element of $T_m$; in other words we do have the union stated in Proposition~\ref{p:cosetdeco}. Finally
claim~\ref{c:cl9} completes the argument by implying that the union is indeed disjoint.

As for the proofs of the claims themselves, they are just routine computations. We only prove Claims~\ref{c:cl1} and~\ref{c:cl2} here; the computations for the other cases are similar. The reader can find them all the  in the longer version of this paper available online~\cite{lflong}.

Claim~\ref{c:cl1} follows from the fact that $U_{-x}A_x^y \equiv JZ_y \pmod{p}$ and $JZ_y=A_{(0,0,0)}^y$.

To prove Claim~\ref{c:cl2}, we first deal with the case $m=0$. For $y \in V, y\neq 0$ let $j= (-\frac{a}{y} +\frac{b}{2}) +\xi \in
(T(\Q_p)\cap K_p)$ (here and elsewhere we interpret $1/\infty =0$). Consider the element $(A_{(0,0,0)}^0)^{-1}
h(l,0)^{-1} j h(l,0) A_{(0,0,0)}^y$. By direct calculation this equals
$$\begin{pmatrix}-\frac{a}{y}&0&0&0\\-c&\frac{-c y^2+a-y b}{y}&0&0\\0&0&-\frac{c y^2+a-y
b}{y}&c\\0&0&0&-\frac{a}{y}\end{pmatrix}$$ if $y\neq \infty$ and equals
$$ \begin{pmatrix}a&0&0&0\\b&-c&0&0\\0&0&c&b\\0&0&0&-a\end{pmatrix}$$ if $y = \infty$. Both these matrices lie in $I_p$ and this
proves the claim for $m=0$.

Now consider $m>0$. For $y \in V, y\neq 0,\infty,$ let $j= cy + p^m\xi \in (T(\Q_p)\cap K_p)$. Consider the
element $(A_{(0,0,0)}^\infty)^{-1}\ h(l,m)^{-1} j\ h(l,m) A_{(0,0,0)}^y$, which by direct calculation equals$$
\begin{pmatrix}-c& \frac{b p^m}{2}&0&0\\c y-\frac{b p^m}{2}&c y^2-\frac{ y p^m
b}{2}+p^{2m} a&0&0\\0&0&-p^{2 m} a-c y^2+\frac{y p^m b}{2}&c y-\frac{b p^m}{2}\\0&0&\frac{b
p^m}{2}&c\end{pmatrix}$$ and this lies in $I_p$. Thus $A_{0,0,0}^y$ is $(l,m)$-equivalent to $A_{0,0,0}^\infty$
and this completes the proof of the claim.

As mentioned already, the proofs of the other claims are similar and hence omitted.
\end{proof}

\subsection{In which we calculate a certain volume}\label{s:volume}

For any $t \in K_p$ we define the volume $I^{l,m}_t$ as follows.
\begin{equation}\label{e:integral}I^{l,m}_t=\text{vol}(R(\Q_p)\backslash R(\Q_p)h(l,m)tI_p).
\end{equation}

In this subsection we shall explicitly compute the volume $I^{l,m}_t$. By Proposition~\ref{p:cosetdeco}, it is
enough to do this for $t\in T_m$. The next two propositions state the results and the rest of the section is
devoted to proving them.

\begin{proposition}\label{p:table}
Let $m > 0$. Let $M_{l,m}$ denote $\frac{p^{3l+4m}}{(p+1)(p^2+1)}$. Then the quantities $I^{l,m}_{t_i}$ for $1
\leq i\leq8$ are as follows.
\begin{align*}I^{l,m}_{t_1}&=pM_{l,m} &I^{l,m}_{t_5}=M_{l,m} \\I^{l,m}_{t_2}&=p^2M_{l,m} &I^{l,m}_{t_6}=pM_{l,m}\\I^{l,m}_{t_3}&=pM_{l,m}&I^{l,m}_{t_7}=p^2M_{l,m}\\I^{l,m}_{t_4}&=M_{l,m}&I^{l,m}_{t_8}=p^3M_{l,m} \end{align*}
\end{proposition}

\begin{proposition}\label{p:table2}
For $m =0$ the quantities $I^{l,m}_t$ are as follows.\\
\begin{align*}I^{l,m}_{t_1}&=\frac{p^{3l +1}}{(p+1)(p^2+1)} &I^{l,m}_{t_5}=\frac{p^{3l}}{(p+1)(p^2+1)} \\I^{l,m}_{t_2}&=\frac{p^{3l+2}}{(p+1)(p^2+1)} &I^{l,m}_{t_7}=\frac{p^{3l+3}}{(p+1)(p^2+1)}\end{align*}

\end{proposition}

\textbf{Remark.} That the volume $I^{l,m}_t$ is finite can be viewed either as a \emph{corollary} of the above propositions, or as a consequence of the fact that $\text{vol}(R(\Q_p)\backslash R(\Q_p)h(l,m)K_p)$ is finite \cite[section 3]{fur}.\\

For each $t\in T_m$ define the subgroup $G_t$ of $K_p$ by
$$G_t = t^{-1}U(\Z_p)GL_2(\Z_p)t \cap I_p
$$
where $U(\Z_p)$ is the subgroup of $K_p$ consisting of matrices that look like $\begin{pmatrix}
1_2& M\\
0 & 1_2\\
\end{pmatrix}$ with $M=  M^T \in M_2(\Z_p)$, and $GL_2(\Z_p)$ (more generally $GL_2(\Q_p)$) is
embedded in  $G(\Q_p)$ via $g \mapsto \begin{pmatrix}
g & 0\\
0 & det(g)\cdot(g^{-1})^T\\
\end{pmatrix}$.

Also let $G_t^1 = tG_tt^{-1}$ be the corresponding subgroup of $U(\Z_p)GL_2(\Z_p).$

And finally, define \begin{equation}\label{def:H}H_t = \{x \in GL_2(\Z_p) \mid \exists y\in U(\Z_p) \text{ such
that } yx \in G_t^1\} .\end{equation} It is easy to see that $H_t =U(\Z_p)G_t^1 \cap GL_2(\Z_p),$ thus $H_t$ is
a subgroup of $GL_2(\Z_p)$.

\begin{lemma} We have a disjoint union $$R(\Q_p)\bs
R(\Q_p)h(l,m)tI_p = \coprod_{y\in G_t \backslash I_p}R(\Q_p)\backslash R(\Q_p)h(l,m)G_t^1ty.$$
\end{lemma}
\begin{proof}
Since $tI_p = \bigcup_{y\in G_t \bs I_p}tG_ty = \bigcup_{y\in G_t \bs I_p}G_t^1ty$, the only thing to prove is
that the union in the statement of the lemma is indeed disjoint.

So suppose that $y_1, y_2$ are two coset representatives of $G_t \bs I_p$ and $rh(l,m)g_1ty_1 =h(l,m)g_2ty_2$
with $g_1 , g_2 \in G_t^1, r\in R(\Q_p).$

This means $ty_2y_1^{-1}t^{-1}$ is an element of $K_p$ that is of the form $\begin{pmatrix}
A& B\\
0 & \text{det}(A)\cdot(A^{-1})^T\\
\end{pmatrix}$. Hence $ty_2y_1^{-1}t^{-1} \in U(\Z_p)GL_2(\Z_p)$.
Thus $y_2y_1^{-1} \in t^{-1}U(\Z_p)GL_2(\Z_p)t \cap I_p = G_t$ which completes the proof.

\end{proof}
By the above lemma it follows that
\begin{align}I^{l,m}_t&=\int_{G_t \bs I_p}dg \cdot
\int_{R(\Q_p)\backslash R(\Q_p)h(l,m)G_t^1}dt\\\label{e:It}&=p^{3(l+m)}[K_p :
I_p]^{-1}[GL_2(\Z_p)U(\Z_p):G_t^1]\int_{R(\Q_p)\backslash R(\Q_p)h(0,m)G_t^1}dt \end{align} \\ where we have
normalized $\int_{U(\Z_p)GL_2(\Z_p) \bs K_p} dx =1.$

On the other hand,
\begin{align*}
R(\Q_p)\backslash R(\Q_p)h(0,m)G_t^1 &=R(\Q_p)\backslash R(\Q_p)h(0,m)U(\Z_p)G_t^1\\ &=R(\Q_p)\backslash
R(\Q_p)h(0,m)(U(\Z_p)G_t^1 \cap GL_2(\Z_p))\\ &= T(\Q_p)\backslash T(\Q_p)h(m)H_t
\end{align*} where $h(m) = \begin{pmatrix}
p^m& 0\\
0 & 1\\
\end{pmatrix}$.

For each $t \in T_m$ let us define
$$A_t=[GL_2(\Z_p)U(\Z_p):G_t^1]$$ and $$V_{t,m}=\int_{T(\Q_p)\backslash
T(\Q_p)h(m)H_t}dt.$$

We use the same normalization of Haar measures as in \cite{fur}, namely we have $$\int_{T(\Q_p)\backslash
T(\Q_p)h(m)GL_2(\Z_p)}dt=1.$$

We summarize the computations above in the form of a lemma.
\begin{lemma}\label{l:keyform}
Let $m \geq 0$. For each $t \in T_m$ we have $$I_t^{l,m} = \frac{p^{3(l+m)}}{(p+1)^2(p^2+1)}\cdot A_t\cdot
V_{t,m}.$$
\end{lemma}
\begin{proof}
This follows from equation~\eqref{e:It}.
\end{proof}

By exactly the same arguments as in \cite[p. 202-203]{fur}, we see that
\begin{equation}\label{e:wtm}
V_{t,m} = [GL_2(\Z_p) : H_t]^{-1}[T(\Z_p) : O_m^t]
\end{equation}
where $O_m^t = T(\Q_p)\cap h(m)H_th(m)^{-1}$.

Let $\Gamma^0_p$ (resp. $\Gamma_{0,p}$) be the subgroup of $GL_2(\Z_p)$ consisting of matrices that become
lower-triangular (resp. upper-triangular) when reduced $\bmod{p}$.
\begin{lemma}\label{l:hti}
\begin{enumerate}

\item We have $H_{t_i} = \Gamma^0_p$ for $i = 1,2,5,8$ and $H_{t_i} =
\Gamma_{0,p}$ for $i = 3,4,6,7.$

\item The quantities $A_{t_i}=
[U(\Z_p)GL_2(\Z_p):G_{t_i}^1]$ are as follows:

\begin{align*}A_{t_1}&=p(p+1) &A_{t_5}=p+1 \\A_{t_2}&=p^2(p+1) &A_{t_6}=p+1\\A_{t_3}&=p^2(p+1)&A_{t_7}=p^3(p+1)\\A_{t_4}&=p(p+1)&A_{t_8}=p^3(p+1) \end{align*}

\end{enumerate}
\end{lemma}
\begin{proof}
We will prove this directly using~\eqref{def:H} and the definition of $A_{t_i}$.

First observe that the cardinality of $U(\F_p)GL_2(\F_p)$ is $p^3 \cdot (p^2-p)(p^2-1) = p^4(p-1)^2(p+1)$.
Recall also that the images of $\Gamma^0_p$ and $\Gamma_{0,p}$ have cardinality $p(p-1)^2$ in $GL_2(\F_p)$.

Suppose $$U = \begin{pmatrix}1 & 0 & n & q\\
0 & 1 & q & r\\
0 & 0 & 1 & 0\\
0 & 0 & 0 & 1\\
\end{pmatrix},G=\begin{pmatrix}a & b & 0 & 0\\
c& d & 0 & 0\\
0 & 0 & d & -c\\
0 & 0 &-b & a\\
\end{pmatrix}.$$

We have $$t_1^{-1}UGt_1 = \begin{pmatrix} a-n d+q b&b&n d-q b&-n c+q a\\c-q d+r b&d&q d-r b&-q c+r a\\a-n d+q
b-d&b&n d-q b+d&-n c+q a-c\\b&0&-b&a\\ \end{pmatrix}.$$ By inspection, this belongs to $I_p$ if and only if $b
\equiv 0 \pmod{p}, n \equiv \frac{a}{d}-1 \pmod{p}$. So $H_{t_1} = \Gamma^0_p$ and $A_{t_1}=
\frac{p^4(p-1)^2(p+1)}{p(p-1)^2 p^2}=p(p+1)$.

Thus we have proved the lemma for the case $t_1$. The proofs of the assertions regarding $t_i$, $2\le i \le 8$ are similar to the above and will be omitted. The reader who wishes to see those details can find them in the longer version available online~\cite{lflong}.

\end{proof}

Let $t$ be such that $H_t = \Gamma^0_p$. Then by working through the definitions, we see that
 \begin{equation}\label{e:om1}O_m^t = x + p^{m+1}y\xi_0, \
x,y\in \Z_p.
\end{equation} On the other hand if $t$ is such that $H_t = \Gamma_{0,p}$,
then we see that \begin{equation}\label{e:om2}O_m^t = x + p^{m}y\xi_0, \ x,y\in \Z_p.
\end{equation}

\begin{lemma}\label{l:Vtm}
Let $m> 0.$ Then we have $V_{t_i,m}= p^m$ for $i=1,2,5,8$ and $V_{t_i,m} = p^{m-1}$ for $i = 3,4,6,7.$
\end{lemma}
\begin{proof}
This follows from \eqref{e:wtm}, \eqref{e:om1}, \eqref{e:om2}, Lemma~\ref{l:hti} and \cite[Lemma 3.5.3]{fur}
\end{proof}

\begin{proof}[Proof of Proposition~\ref{p:table}]
The proof is a consequence of Lemma~\ref{l:keyform}, Lemma~\ref{l:hti} and Lemma~\ref{l:Vtm}.
\end{proof}

Let us now look at the case $m=0$. In this case $T_0 = \{t_1,t_2,t_5,t_7\}$.

The groups $H_{t_i}$ and the quantities $[GL_2(\Z_p) : H_{t_i}]^{-1}$ have already been calculated. On the other
hand we now have
\begin{equation}\label{e:o0t}
O_0^{t_i} = x + py\xi_0, \ x,y\in \Z_p.
\end{equation}
for each $t_i \in T_0$.

\begin{proof}[Proof of Proposition~\ref{p:table2}]
We have already calculated each $A_{t_i}$. Also by \eqref{e:wtm},\eqref{e:o0t} and Lemma~\ref{l:hti} we conclude
that each $V_{t_i,0}=1$. Now the result follows as before, from Lemma~\ref{l:keyform}.

\end{proof}
\subsection{Simplification of the local zeta integral}

 Recall the definition of the key local integral $Z_p(s)$ from section~\ref{s:rankin}. In \eqref{e:sumdeco} we reduced this integral to an useful sum. Now suppose that $W_p$ and $B_p$ are right $I_p$-invariant. Then proposition~\ref{p:cosetdeco} allows us to further simplify that expression as follows.
 \begin{equation}\label{e:keyformula}
Z_p(s) = \sum_{l\in \Z, m\geq 0}\sum_{t \in T_m} W_p(\Theta h(l,m)t,s)\cdot B_p(h(l,m)t)\cdot I_t^{l,m}
\end{equation}

Note that in the above formula we mildly abuse notation and use $\Theta$ to really mean its natural inclusion in $\G(\Q_p)$. We will continue to do this in the future for notational economy.

\section{The evaluation of the local Bessel functions in
the Steinberg case} \label{s:appendix}
 \subsection{Background}

Because automorphic representations of $GSp(4)$ are not necessarily generic, the Whittaker model is not always
useful for studying L-functions. For many problems, the Bessel model is a good substitute. Explicit evaluation
of local zeta integrals then often reduces to explicit evaluation of certain local Bessel functions. 

Formulas
for local Bessel functions at the non-archimedean places have been established in the following cases:

\begin{itemize}
\item \emph{unramified} representations of $GSp_4(\Q_p)$ \cite{sug},

\item \emph{unramified} representations (the Casselman-Shalika like formula)~\cite{bff},

\end{itemize}

Formulas
for the archimedean Bessel functions have been established in the following cases:
\begin{itemize}
\item \emph{class-one} representations on $Sp_4(\R)$~\cite{niw},

\item large \emph{discrete series} and $P_J$-\emph{principal series} of $Sp_4(\R)$~\cite{miy},

\item \emph{principal series} of $Sp_4(\R)$~\cite{ish}.

\end{itemize}

In this section we give an explicit formula for the Bessel function for an unramified quadratic twist of the
\emph{Steinberg} representation of $GSp_4(\Q_p)$. By \cite{sch} this is precisely the
 representation corresponding to a local newform for
the Iwahori subgroup.

Throughout this section we let $p$ be an odd prime that is inert in $L$.  We suppose that the local component
$(\omega_\pi)_p $ is trivial, the conductor of  $\psi_p$ is $\Z_p$ and $S =
\begin{pmatrix} a & b/2 \\ b/2 & c\\ \end{pmatrix} \in M_2(\Z_p)$.

Because $p$ is inert, $L_p$ is a quadratic extension of $\Q_p$ and we may write elements of $L_p$ in the form $a + b
\sqrt{-d}$ with $a,b \in \Q_p$; then $\Z_{L,p} = a +b \sqrt{-d}$ where $a,b \in \Z_p.$  We identify $L_p$ with $T(\Q_p)$ and $\xi$ with $\sqrt{-d}/2$. Then $T(\Z_p) = \Z_{L,p}^\times$ consists of elements of the form $a +b \sqrt{-d}$ where $a,b$ are elements of $\Z_p$ not both divisible by $p$.

We assume that $\Lambda_p$ is trivial on the elements of $T(\Z_p)$ of the form $a + b\sqrt{-d}$ with $a,b \in \Z_p, p \mid b,
p \nmid a$. Further, we assume that $\Lambda_p$ is \emph{not} trivial on the full group $T(\Z_p)$, that is, it is not unramified.

Finally, assume that the local representation $\pi_p$ is an unramified twist of the Steinberg representation. This is
representation \textrm{IV}a in \cite[Table 1]{sch}. The space of $\pi_p$ contains a unique normalized vector
that is fixed by the Iwahori subgroup $I_p$. We can think of this vector as the normalized local newform for
this representation.
\subsection{Bessel functions}\label{s:besmodelcalc}
Let $\B$ be the space of locally constant functions $\varphi$ on $G(\Q_p)$ satisfying $$\varphi(tuh)=
\Lambda_p(t)\theta_p(u)\varphi(h), \text{   for } t\in T(\Q_p),u \in U(\Q_p), h \in G(\Q_p).$$ Then by
Novodvorsky and Piatetski-Shapiro \cite{nov}, there exists a unique subspace $\B(\pi_p)$ of $\B$ such that the
right regular representation of $G(\Q_p)$ on  $\B(\pi_p)$ is isomorphic to $\pi_p$. Let $B_p$ be the unique
$I_p$-fixed vector in $\B(\pi_p)$ such that $B_p(1_4) =1$. Therefore
\begin{equation}\label{e:bformula}B_p(tuhk)= \Lambda_p(t)\theta_p(u)\varphi(h),\end{equation} where  $t\in
T(\Q_p),u \in U(\Q_p), h \in G(\Q_p), k\in I_p$.

Our goal is to explicitly compute $B_p$. By Proposition~\ref{p:cosetdeco} and ~\eqref{e:bformula} it is enough
to compute the values $B_p(h(l,m)t_i)$ for $l\in \Z, m\in \Z_{\ge 0}, t_i \in T_m.$

Let us fix some notation. Recall the matrices $t_i$ which were defined in Subsection~\ref{s:intsum}. Also we
will frequently use other notation from Section~\ref{s:strategypadic}. We now define

\begin{align*}a_0^{l,m}&=B_p(h(l,m)t_7),   &a_{\infty}^{l,m}=B_p(h(l,m)t_8), \\
b_0^{l,m}&=B_p(h(l,m)t_2),    &^1b_0^{l,m}=B_p(h(l,m)t_1),\\
 b_{\infty}^{l,m}&=B_p(h(l,m)t_3), &^1b_{\infty}^{l,m}=B_p(h(l,m)t_4),
\\c_0^{l,m}&=B_p(h(l,m)t_5), &c_{\infty}^{l,m}=B_p(h(l,m)t_6).\end{align*}

\begin{lemma}Let $ m\geq 0, y \in \{0, \infty \}$. The following equations hold: \begin{enumerate}
\item $a_y^{l,m} = 0\quad $ if  $l<-1.$
\item $^1b_0^{l,m}=b_0^{l,m} =\ ^1b_{\infty}^{l,0}=b_{\infty}^{l,0}=  0 \quad$ if  $l<0.$
\item $^1b_{\infty}^{l,m}=b_{\infty}^{l,m}=  0 \quad$ if  $l<-1.$
\item $c_y^{l,m} = 0\quad$  if  $l <0.$
\end{enumerate}

\end{lemma}
\begin{proof}
First note that $U_{(0,0,p)}t_i \equiv t_i \pmod{p}$, hence they are in the same coset of $K_p / I_p$. Hence
\begin{align*}B_p(h(l,m)t_i) &= B_p(h(l,m)U_{(0,0,p)}t_i)\\ &= B_p(U_{(0,0,p^{l+1})}h(l,m)t_i)\\ &=\psi_p(p^{l+1}c)B_p(h(l,m)t_i).\end{align*}
Since the conductor of  $\psi_p$ is $\Z_p$ and $c$ is a unit, it follows that $B_p(h(l,m)t_i)=0$ for $l<-1$.
This completes the proof of (a) and (c).

Next, observe that \begin{align*}c_y^{l,m} &= B_p(h(l,m)Z_y)\\
&=B_p(h(l,m)U_{(0,0,1)}Z_y)\\
&=B_p(U_{(0,0,p^l)}h(l,m)Z_y)\\
&=\psi_p(p^{l}c)B_p(h(l,m)Z_y).\end{align*} It follows that $B_p(h(l,m)Z_y)=0$ for $l<0$. This completes the
proof of (d).

Next, we have \begin{align*}B_p(h(l,m)JU_{(1,0,0)}JZ_y) &= B_p(h(l,m)JU_{(1,0,0)}JU_{0,0,1}Z_y)\\
&=B_p(h(l,m)U_{0,0,1}JU_{(1,0,0)}JZ_y)\\
&=\psi_p(p^{l}c)B_p(h(l,m)JU_{(1,0,0)}JZ_y).\end{align*} It follows that $^1b_0^{l,m}=b_0^{l,m} = 0$ for $l<0$.

Finally, \begin{align*}B_p(h(l,0)JU_{(0,0,1)}JZ_y) &= B_p(h(l,0)JU_{(0,0,1)}JU_{1,0,0}Z_y)\\
&=B_p(h(l,0)U_{1,0,0}JU_{(0,0,1)}JZ_y)\\
&=\psi_p(p^{l}a)B_p(h(l,0)JU_{(0,0,1)}JZ_y).\end{align*} It follows that $^1b_\infty^{l,0}=b_\infty^{l,0} = 0$
for $l<0$. This completes the proof of (b).

\end{proof}

By our normalization, we have $c_0^{0,0} = 1$. From Proposition~\ref{p:cosetdeco}, proof of Claim 6, it follows
that $c_\infty^{0,0} = \Lambda_p(\frac{b+\sqrt{-d}}{2})$.

To get more information, we have to use the fact that the local  Iwahori-Hecke algebra acts on $B_p$ in a
precise manner.
\subsection{Hecke operators and the results}

Henceforth we always assume that $l \geq -1, m\geq 0$. In particular, all equations that are stated without
qualification will be understood to hold in the above range. We know that $\pi_p$ is either $\text{St}_{GSp(4)}$
or $\xi_0 \text{St}_{GSp(4)}$ where $\xi_0$ is the non-trivial unramified quadratic character.  Put $w_p = -1$
in the former case and $w_p=1$ in the latter. Put $$\eta_p =
\begin{pmatrix}0&0&0&1\\0&0&1&0\\0&p&0&0\\p&0&0&0\end{pmatrix}.$$ Also, for $y \in V$, define the matrices $R_y$
as follows: If $y \in Y$, $$R_y = (U_{(y,0,0)})^T,$$ and $$R_\infty =
\begin{pmatrix}0&0&-1&0\\0&-1&0&0\\-1&0&0&0\\0&0&0&1\end{pmatrix}.$$

Let $t\in G(\Q_p)$. By \cite{sch}, we know the following:
\begin{equation}\label{e:heckeformula1}\sum_{y\in V}B_p(tZ_y) =0, \end{equation}
\begin{equation}\label{e:heckeformula2}B_p(t\eta_p) = w_p B_p(t), \end{equation}
\begin{equation}\label{e:heckeformula3}\sum_{y\in V}B_p(tR_y) =0. \end{equation}

\eqref{e:heckeformula1} and Proposition~\ref{p:cosetdeco} immediately imply
\begin{align}\label{e:heckeformula4}a_0^{l,m} + pa_\infty^{l,m} &= 0, &\text{for } m>0\\ \label{e:heckeformula6}
pb_y^{l,m} +\ ^1b_y^{l,m} &= 0, &\text{for } y \in \{0,\infty\}\\ \label{e:heckeformula5} pc_0^{l,m} +
c_\infty^{l,m} &= 0, &\text{for } m>0
\end{align}

Next we act upon by $\eta_p$. Check that $$(h(l+1,m)B_{(0,0,0)}^\infty)^{-1}h(l,m)A_{(0,0,0)}^0\eta_p= \begin{pmatrix}1&0&0&0\\0&1&0&0\\0&0&-1&0\\0&0&0&-1\end{pmatrix}.$$  So we have \begin{align*}a_0^{l,m}&=B_p(h(l,m)A_{(0,0,0)}^0)\\
&=w_pB_p(h(l,m)A_{(0,0,0)}^0\eta_p)\\
&=w_pB_p(h(l+1,m)B_{(0,0,0)}^\infty).\end{align*}

Thus \begin{equation}a_0^{l,m}= w_pc_\infty^{l+1,m}.\end{equation}

We also have   $$(h(l+1,m)B_{(0,0,0)}^0)^{-1}h(l,m)A_{(0,0,0)}^\infty\eta_p=
\begin{pmatrix}1&0&0&0\\0&1&0&0\\0&0&-1&0\\0&0&0&-1\end{pmatrix}.$$ So similarly, we conclude
\begin{equation}\label{e:heckeformula8}a_\infty^{l,m}= w_pc_0^{l+1,m}.\end{equation}

Next, check that $$(h(l,m)B_{(1,0,0)}^1 \eta_p)^{-1}h(l-1,m+1)U_{(-1/p,0,0)}D_\infty^1 = (Z^1)^T \in I_p.$$
Hence $$B_p(h(l,m)B_{(1,0,0)}^1)= w_pB_p(h(l-1,m+1)D_\infty^1).$$ (Note that both sides are zero if $l=-1,
m=0$).

By the proof of Proposition~\ref{p:cosetdeco}, $B_p(h(l,m)B_{(1,0,0)}^1) = b_0^{l,m}$ and
$B_p(h(l-1,m+1)D_\infty^1)=\psi_p(p^{l-1} c)b_\infty^{l-1,m+1}.$

Thus we have proved \begin{equation}\label{e:heckeformula9}b_0^{l,m}=w_p\psi_p(p^{l-1}
c)b_\infty^{l-1,m+1}.\end{equation}

At this point we pause and note that on account of~\eqref{e:heckeformula4}--\eqref{e:heckeformula9} it is enough
to compute the quantities $b_\infty^{l,m} , a_0^{l,m}, l\ge-1, m\ge 0, l+m\neq-1$. Of course, we already know
that $a_0^{-1,0} = w_p\Lambda_p(\frac{b+\sqrt{-d}}{2})$.

Next, we use ~\eqref{e:heckeformula3}.

For each $x\in Y$, we can check that $A_{0,0,0}^0R_x = A_{-x,0,0}^0$. Furthermore, $A_{0,0,0}^0R_\infty =
D_\infty^0$. Assuming $l+m \ge0$ we have $B_p(h(l,m)A_{-x,0,0}^0) = a_0^{l,m}$ and $B_p(h(l,m)D_\infty^0=
\psi_p(p^lc)b_\infty^{l,m}$. So using~\eqref{e:heckeformula3} we conclude
\begin{equation}pa_0^{l,m}=-\psi_p(p^lc)b_\infty^{l,m},\end{equation} for $l+m \ge 0$.

However we can do more. Check that for $x\in Y$, $A_{(0,0,0)}^\infty R_x= A_{(0,0,-x)}^\infty$ and
$A_{(0,0,0)}^\infty R_\infty \equiv D_0^0 \pmod{p}$. If $l\ge0$ we have $B_p(h(l,m)A_{(0,0,-x)}^\infty =
a_\infty^{l,m}$ and $B_p(h(l,m)D_0^0)=b_0^{l,m}$. So again using~\eqref{e:heckeformula3} we have
\begin{equation}\label{e:heckeformula10}pa_\infty^{l,m}=-b_0^{l,m},\end{equation} for $l \ge 0$.

So~\eqref{e:heckeformula4}, \eqref{e:heckeformula9} and \eqref{e:heckeformula10} imply that for $l\ge0$,
$m>0$\begin{equation}\label{e:heckeformula11}b_\infty^{l,m} = -pb_0^{l,m} =
-pw_p\psi_p(p^{l-1}c)b_\infty^{l-1,m+1}.\end{equation}

Now observe that $B_{0,0,0}^0 R_\infty \equiv D_0^\infty \pmod{p}$ and for $x\in Y$, $B_{0,0,0}^0 R_x =
B_{-x,0,0}^0.$ Assuming $l+m \neq -1$ we have $B_p(h(l,m)D_0^\infty) =\ ^1b_0^{l,m}$ and for $x\in y, x\neq 0$,
$B_p(h(l,m)B_{-x,0,0}^0) =\ ^1b_0^{l,m}$. Hence using~\eqref{e:heckeformula3}\begin{equation}c_0^{l,m} = -p\
^1b_0^{l,m}\end{equation} So by equations~\eqref{e:heckeformula6} and~\eqref{e:heckeformula8} we have,
\begin{equation}\label{e:heckeformula115}a_\infty^{l,m}=p^2\psi_p(p^{l}c)b_\infty^{l,m+1}\end{equation} The
above equation, along with our normalization tells us that
\begin{equation}\label{e:heckeformula12}b_\infty^{-1,1}=\frac{1}{p^2}\psi_p(-\frac{c}{p})w_p.\end{equation}
Also, using \eqref{e:heckeformula10},\eqref{e:heckeformula11} and \eqref{e:heckeformula115} we get
\begin{equation}\label{e:hecke16}b_\infty^{l,m+1} = \frac{1}{p^4}b_\infty^{l,m}\end{equation} for $l\ge 0, m>0.$

\eqref{e:heckeformula11},\eqref{e:hecke16} and\eqref{e:heckeformula12} imply :
\begin{equation}\label{e:starth}b_\infty^{l,m}=-\frac{(-pw_p)^{l}}{p^{4l+4m+1}}\text{ \ if }l\ge0, m\ge 1\end{equation}
\begin{equation}b_\infty^{-1,m}=\frac{1}{p^{4m-2}} \psi_p(-\frac{c}{p})w_p \text{ \ if } m\ge1.\end{equation}

In the case $m=0$, Proposition~\ref{p:cosetdeco}, proof of Claim 7, tells us that $\ ^1b_\infty^{l,0} =
\Lambda_p(\frac{b+\sqrt{-d}}{2}) \ ^1b_0^{l,0}$ which implies \begin{equation}\label{e:endh}b_\infty^{l,0} =
\Lambda_p(\frac{b+\sqrt{-d}}{2}) b_0^{l,0}=w_p\psi_p(p^{l-1}c)\Lambda_p(\frac{b+\sqrt{-d}}{2})
b_\infty^{l-1,1}\end{equation}

Equation~\eqref{e:starth}--\eqref{e:endh}, along with the earlier equations that specify the inderdependence of
various quantities, determine all the values $B_p(h(l,m)t_i).$ For convenience, we compactly state the facts proven above
 above as two propositions. We only state it for $l\ge 0$ since that is the only case needed for our later
applications. The values for $l=-1$ can be easily gleaned from these and the above equations.

\begin{proposition}\label{p:besselvalues1}Let $l\ge 0, m>0$. Put $M=(-pw_p)^l p^{-4(l+m)}$. Then the following hold:
\begin{enumerate}
\item $B_p(h(l,m)t_1)=M\cdot\frac{-1}{p},$
\item $B_p(h(l,m)t_2)=M\cdot\frac{1}{p^2},$
\item $B_p(h(l,m)t_3)=M\cdot\frac{-1}{p},$
\item $B_p(h(l,m)t_4)=M$
\item $B_p(h(l,m)t_5)=M$
\item $B_p(h(l,m)t_6)=M\cdot(-p),$
\item $B_p(h(l,m)t_7)=M\cdot\frac{1}{p^2},$
\item $B_p(h(l,m)t_8)=M\cdot\frac{-1}{p^3}.$
\end{enumerate}
\end{proposition}

\begin{proposition}\label{p:besselvalues2} Let $l\ge0$. Put $M=(-pw_p)^l p^{-4l}$. Then the following hold:
\begin{enumerate}
\item $B_p(h(l,0)t_1)=M\cdot\frac{-1}{p},$
\item $B_p(h(l,0)t_2)=M\cdot\frac{1}{p^2},$
\item $B_p(h(l,0)t_5)=M,$
\item $B_p(h(l,m)t_7)=M\cdot\frac{-\Lambda_p(\frac{b+\sqrt{-d}}{2})}{p^3}.$
\end{enumerate}
\end{proposition}

\section{The case unramified $\pi_p$, Steinberg $\sigma_p$ }

\subsection{Assumptions}
Suppose that the characters $\omega_\pi , \omega_\sigma, \chi_0$ are trivial. Let $p \neq 2$ be a finite prime
of $\Q$ such that

\begin{enumerate}

\item $p$ is inert in $L =\Q(\sqrt{-d})$.

\item\label{i:lambda1} The local components $\Lambda_p$ and $\pi_p$ are unramified.

\item $\sigma_p$ is the Steinberg representation (or its twist by the unramified quadratic
character).

\item The conductor of  $\psi_p$ is
 $\Z_p$

\item $S = \begin{pmatrix} a & b/2 \\ b/2 & c\\ \end{pmatrix} \in
M_2(\Z_p)$.

\item $-d = b^2 - 4ac$ generates the discriminant of $L_p/ \Q_p.$

\end{enumerate}

\subsection{Description of $B_p$ and $W_p$}\label{s:bpwp1}
Given the local representations and characters as above, define $I(\Pi_p,s)$ and the local Bessel and Whittaker
spaces as in Sections 1 and 2. For any choice of local Whittaker and Bessel functions $W_p$ and $B_p$ we can define
the local zeta integral $Z_p(s)$ by \eqref{e:local}. We now fix such a choice.

As in the unramified case from section~\ref{prelim}, we let $B_p$ be the unique normalized $K_p$-vector in the
local Bessel space. Sugano\cite{sug} has computed the function $B_p$ explicitly.

We now define $W_p$. Let $\widetilde{U_p}$ be the subgroup of $\widetilde{K_p}$ defined by $$\widetilde{U_p} =
\big\{ z \in \widetilde{K_p} \mid z \equiv
\begin{pmatrix}
\ast & 0 & \ast & \ast\\
\ast & \ast & \ast & \ast\\
\ast & 0 & \ast & \ast\\
0 & 0 & 0  & \ast\\
\end{pmatrix} \pmod{p} \big\}.$$ It is not hard to see that $I(\Pi_p, s)$ has $\widetilde{U_p}$-fixed vectors. Now let
 $W_p$ be the unique $\widetilde{U_p}$-fixed vector in the local Whittaker space with the
 following properties:
 \begin{itemize}
 \item $W_p(e,s)=1$,
 \item $W_p(g,s)=0$ if $g$ does not belong to $P(\Q_p)\widetilde{U_p}$
 \end{itemize}

Concretely we have the following description of $W_p(\,s).$

We know that $\sigma_p = \text{Sp} \otimes \tau$ where $Sp$ denotes the special (Steinberg) representation and
$\tau$ is a (possibly trivial) unramified quadratic character. We put $a_p = \tau(p)$, thus $a_p = \pm 1$ is the
eigenvalue of the local Hecke operator $T(p)$.

Let $W_p'$ be the unique function on $GL_2(\Q_p)$ such that
\begin{equation}\label{e:beginf2}
W_p'(gk) = W_p'(g), \text{for } g \in GL_2(\Q_p), k\in \Gamma_{0,p},
\end{equation}
\begin{equation}
W_p'(\begin{pmatrix}1 &x\\0&1\end{pmatrix}g) = \psi_p(-cx)W_p'(g), \text{for } g \in GL_2(\Q_p), x\in \Q_p,
\end{equation}
\begin{equation}
W_p'\begin{pmatrix}a &0\\0&1\end{pmatrix} = \begin{cases} \tau(a) |a| &\text{if } | a|_p \leq1, \\ 0&
\text{otherwise} \end {cases}
\end{equation}

\begin{equation}
W_p'\left(\begin{pmatrix}a &0\\0&1\end{pmatrix} \begin{pmatrix} 0 & 1 \\
-1& 0 \end{pmatrix} \right)=
\begin{cases} -p^{-1}\tau(a) |a| &\text{if } | a|_p \leq p, \\ 0&
\text{otherwise} \end {cases}
\end{equation}

We extend $W_p'$ to a function on $GU(1,1)(\Q_p)$ by
$$W_p'(ag)=W_p'(g), \text{ for } a\in L_p^\times, g\in GL_2(\Q_p).$$

Then, $W_p( \,s)$ is the unique function on $\G(\Q_p)$ such that
\begin{equation}
W_p(mnk,s) = W_p(m,s), \text{ for } m\in M(\Q_p), n\in N(\Q_P), k\in \widetilde{U_p},
\end{equation}

\begin{equation}
W_p(e)= 1 \text{ and } W_p(g,s)=0 \text{ if } g \notin P(\Q_p)\widetilde{U_p},
\end{equation}
and
\begin{equation}\label{e:endf2}\begin{split}
W_p\left(\begin{pmatrix} a&0&0&0\\0&1&0&0\\0&0&\overline{a}^{-1}&0\\0&0&0&1\end{pmatrix}\begin{pmatrix}
1&0&0&0\\0&a_1&0&b_1\\0&0&c_1&0\\0&d_1&0&e_1\end{pmatrix},s\right)
 \\= \left|N_{L/\Q}(a)\cdot c_1^{-1} \right|_p^{3(s+1/2)}\cdot
 W_p'\begin{pmatrix}a_1&b_1\\d_1&e_1\end{pmatrix}\end{split}
 \end{equation}
for $a\in \Q_p^\times,
\begin{pmatrix}a_1&b_1\\d_1&e_1\end{pmatrix}\in GU(1,1)(\Q_p), c_1 =
\mu_1\begin{pmatrix}a_1&b_1\\d_1&e_1\end{pmatrix}$.

Let us use the following notation: For $\begin{pmatrix}a&b\\c&d\end{pmatrix} \in GU(1,1)$ we let
$$m^{(2)}(\begin{pmatrix}a&b\\c&d\end{pmatrix} ) = \begin{pmatrix}1 &0&0&0\\0&a&0&b\\0&0&\beta&0\\0&c&0&d
\end{pmatrix}$$ where $\beta = \mu_1(\begin{pmatrix}a&b\\c&d\end{pmatrix} )$.

\subsection{The results}

For $i= 1,2,3,4$, define the characters $\gamma_p^{(i)}$ of $\Q_p^\times$ as in Section~\ref{prelim}. We now
state and prove the main theorem of this section.
\begin{theorem}\label{t:unramifiedsteinberg}
Let the functions $B_p, W_p$ be as defined in subsection~\ref{s:bpwp1}. Then we have
 $$Z_p(s,W_p,B_p) = \frac{1}{p^2+1} \cdot \frac{L(3s + \frac{1}{2} , \pi_p
\times \sigma_p)}{L(3s + 1, \sigma_p \times \rho(\Lambda_p))}$$ where,
 $$L(s , \pi_p \times
\sigma_p)= \prod_{i=1}^4(1- \gamma_p^{(i)}(p)a_pp^{-1/2}p^{-s})^{-1},$$ and $$L(s, \sigma_p \times
\rho(\Lambda_p)) = (1-p^{-2s-1})^{-1}.$$

\end{theorem}

Before we begin the proof, we need a lemma.

\begin{lemma}\label{l:whittakerramified} We have the following formulae for $W_p(\Theta h(l,m)t_i,s)$ where $t_i \in T_m.$

\begin{enumerate}

\item If $m>0$ then $W_p(\Theta h(l,m)t_i,s) =\begin{cases} p^{-6ms-3ls-3m-5l/2}a_p^l & \text{if } i \in \{1,5\} \\
p^{-6ms-3ls-3m-5l/2}a_p^l\cdot \frac{-1}{p} & \text{if } i \in \{3,7\}\\0 & \text{otherwise}\end{cases}$

\item $W_p(\Theta h(l,0)t_i,s) =\begin{cases} p^{-3ls-5l/2}a_p^l & \text{ if } i \in \{1,5\} \\
0 & \text{ if }  i \in \{2,7\} \end{cases}$
\end{enumerate}
\end{lemma}

\begin{proof}
We have \begin{equation}\label{e:thetausedlater}\Theta h(l,m) = h(l,m)\begin{pmatrix}1
&0&0&0\\p^m\alpha&1&0&0\\0&0&1&-p^m\overline{\alpha}\\0&0&0&1
\end{pmatrix}\end{equation}

First consider the case $m > 0.$

We can check that $\Theta h(l,m)t_i \notin P(\Q_p)\widetilde{U_p}$ if $i \in \{2,4,6,8\}.$

For the remaining $t_i$ (i.e. $i \in \{1,3,5,7\}$) we have the following decompositions:

$\Theta h(l,m)t_1$ $$= \begin{pmatrix}p^{2m+l} &0&0&0\\0&1&0&0\\0&0&p^{-2m-l}&0\\0&0&0&1
\end{pmatrix}m^{(2)}(\begin{pmatrix}p^{m+l}&0\\0&p^m\end{pmatrix} )\begin{pmatrix}-1&0&0&0\\-p^m \alpha&-1&0&0\\1&0&-1&p^m
\overline{\alpha}\\0&0&0&-1\end{pmatrix}$$

$\Theta h(l,m)t_3$ = $$\begin{pmatrix}p^{2m+l} &0&0&0\\0&1&0&0\\0&0&p^{-2m-l}&0\\0&0&0&1
\end{pmatrix}m^{(2)}(\begin{pmatrix}p^{m+l}&0\\-p^m&p^m\end{pmatrix} )\begin{pmatrix}-1&0&0&0\\-p^m \alpha&-1&0&0\\0&-p^m \overline{\alpha}&-1&p^m
\overline{\alpha}\\-p^m \alpha&0&0&-1\end{pmatrix}$$

 $\Theta h(l,m)t_5$ $$= \begin{pmatrix}p^{2m+l}
&0&0&0\\0&1&0&0\\0&0&p^{-2m-l}&0\\0&0&0&1
\end{pmatrix}m^{(2)}(\begin{pmatrix}p^{m+l}&0\\0&p^m\end{pmatrix} )\begin{pmatrix}1&0&0&0\\p^m \alpha&1&0&0\\0&0&1&p^m
\overline{\alpha}\\0&0&0&1\end{pmatrix}$$

 $\Theta h(l,m)t_7$ $$ = \begin{pmatrix}p^{2m+l}
&0&0&0\\0&1&0&0\\0&0&p^{-2m-l}&0\\0&0&0&1
\end{pmatrix}m^{(2)}(\begin{pmatrix}0&-p^{m+l}\\-p^m&0\end{pmatrix} )\begin{pmatrix}0&0&1&0\\0&1&0&0\\-1&p^m \overline{\alpha}&0&0\\0&0&p^m\alpha&1\end{pmatrix}$$

Part (a) of the lemma now follows from the above decompositions and equations\eqref{e:beginf2}-\eqref{e:endf2}.

Let us now look at $m=0$. Once again, check that $\Theta h(l,0)t_i \notin P(\Q_p)\widetilde{U_p}$ if $i \in \{2,7\}.$ For $t_1$ and $t_5$ we have the above
decompositions, from which part (b) follows via the equations \eqref{e:beginf2}-\eqref{e:endf2}.

\end{proof}

\begin{proof}[Proof of Theorem~\ref{t:unramifiedsteinberg}]
By \eqref{e:keyformula}  we have

\begin{equation}\label{e:wp1t2}
Z_p(s,W_p,B_p) = \sum_{l\geq 0, m \geq 0}B_p( h(l,m))\sum_{t_i \in T_m} W_p(\Theta h(l,m)t_i,s)\cdot
I_{t_i}^{l,m}
\end{equation}

We first look at the terms corresponding to $m>0.$ From Lemma~\ref{l:whittakerramified} and
Proposition~\ref{p:table} we have $\sum_{t_i \in T_m} W_p(\Theta h(l,m)t_i , s)\cdot I_{t_i}^{l,m} = 0$. So only
terms corresponding to $m=0$ contribute.

From Proposition~\ref{p:table2} and Lemma~\ref{l:whittakerramified} we have $$\sum_{t_i \in T_0} W_p(\Theta
h(l,0)t_i,s)\cdot I_{t_i}^{l,0} = \frac{1}{p^2+1}\cdot p^{-3ls+l/2}a_p^l.$$

Hence \eqref{e:wp1t2} reduces to $$Z_p(s,W_p,B_p) = \frac{1}{p^2+1} \cdot\sum_{l\geq 0}B_p(
h(l,0))p^{-3ls+l/2}a_p^l.$$ Define $C(y) = \sum_{l \geq 0}B_p( h(l,0))y^l$. We are interested in the quantity
\begin{equation}\label{e:zpramified} Z_p(s,W_p,B_p)=\frac{1}{p^2+1}C(a_p p^{-3s
+1/2}).\end{equation}

Sugano, in \cite[p. 544]{sug}, has computed $C(y)$ explicitly. He proves that $$C(y) =
\frac{H(y)}{Q(y)}$$ where $H(y) = 1 - \frac{y^2}{p^4}, Q(y) = \prod_{i=1}^4(1-\gamma_p^{(i)}(p)p^{-3/2}y).$

Plugging in these values in \eqref{e:zpramified} we get the desired result.

\end{proof}

\section{The case Steinberg $\pi_p$, Steinberg $\sigma_p$ }

\subsection{Assumptions}
Suppose that the characters $\omega_\pi , \omega_\sigma, \chi_0$ are trivial. Let $p \neq 2$ be a finite prime
of $\Q$ such that

\begin{enumerate}

\item $p$ is inert in $L =\Q(\sqrt{-d})$.

\item\label{i:lambda} $\Lambda_p$ is not trivial on $T(\Z_p)$; however it is trivial on $T(\Z_p)\cap \Gamma^0_p$.

\item $\pi_p$ is the Steinberg representation (or its twist by the unique non-trivial unramified quadratic character).

\item $\sigma_p$ is the Steinberg representation (or its twist by the unique non-trivial unramified quadratic
character).

\item The conductor of  $\psi_p$ is
 $\Z_p.$

\item $S = \begin{pmatrix} a & b/2 \\ b/2 & c\\ \end{pmatrix} \in
M_2(\Z_p)$.

\item $-d = b^2 - 4ac$ generates the discriminant of $L_p/ \Q_p.$

\end{enumerate}

\textbf{Remark}. $\pi_p$ corresponds to a local newform for the Iwahori subgroup $I_p$ (see \cite{sch}). Also, as in the previous section, $\sigma_p$ corresponds to the
local newform for the Iwahori subgroup $\Gamma_0(p)$ of $GL_2(\Q_p)$.

\subsection{Description of $B_p$ and $W_p$}\label{s:bpwp2}
Let $\Phi_p$ be the unique normalized local newform for the Iwahori subgroup $I_p$, as defined by
Schmidt~\cite{sch}. Let $w_p$ be the local Atkin-Lehner eigenvalue for $\pi_p$; this equals $-1$ when $\pi_p$ is
the Steinberg representation and equals $1$ when $\pi_p$ is the unramified quadratic twist of the Steinberg
representation. We let $B_p$ be the normalized vector that corresponds to $\Phi_p$ in the Bessel space.
Section~\ref{s:appendix} was devoted to the computation of the values $B_p(h(l,m)t)$ for $l, m \in \Z, m
\geq 0, t\in T_m.$\\

Because $p$ is inert, $L_p$ is a quadratic extension of $\Q_p$ and we may write elements of $L_p$ in the form $a + b
\sqrt{-d}$ with $a,b \in \Q_p$; then $\Z_{L,p} = a +b \sqrt{-d}$ where $a,b \in \Z_p.$  We also identify $L_p$ with $T(\Q_p)$ and $\xi$ with $\sqrt{-d}/2$. We now define $W_p$. By Assumption~\eqref{i:lambda} above, we have $\Lambda_p$ is trivial on the elements of $T(\Q_p)$ of the form $a + b\sqrt{-d}$ with $a,b \in \Z_p, p \mid b,
p \nmid a$. Take the canonical map $r: \widetilde{K}_p \rightarrow \widetilde{G}(\F_p)$ and define
$I_p' = r^{-1}(I(\F_p))$, where $I(\F_p)$ is the subgroup of $G(\F_p)$ defined in the beginning of this paper.

Let $s_1$ denote the matrix $\begin{pmatrix}0&1&0&0\\1&0&0&0\\0&0&0&1\\0&0&1&0\end{pmatrix}.
$

Let $W_p(\ ,s)$ be the unique vector in $I(\Pi_p,s)$ with the
 following properties:
 \begin{itemize}
 \item $W_p(1,s)=1$,
 \item $W_p(s_1,s)=1$,
 \item $W_p(gk,s)=W_p(g,s)$ is $k \in I_p'$,
 \item $W_p(g,s)=0$ if $g$ does not belong to $P(\Q_p)I_p' \sqcup P(\Q_p)s_1I_p' $
 \end{itemize}

Concretely we have the following description of $W_p( \ ,s):$

We know that $\sigma_p = \emph{Sp} \otimes \tau$ where $\emph{Sp}$ denotes the special (Steinberg) representation and
$\tau$ is a (possibly trivial) unramified quadratic character. We put $a_p = \tau(p)$, thus $a_p = \pm 1$ is the
eigenvalue of the local Hecke operator $T(p)$.

Let $W_p'$ be the unique function on $GL_2(\Q_p)$ such that
\begin{equation}\label{e:beginf3}
W_p'(gk) = W_p'(g), \text{for } g \in GL_2(\Q_p), k\in \Gamma_{0,p},
\end{equation}
\begin{equation}
W_p'(\begin{pmatrix}1 &x\\0&1\end{pmatrix}g) = \psi_p(-cx)W_p'(g), \text{for } g \in GL_2(\Q_p), x\in \Q_p,
\end{equation}
\begin{equation}
W_p'\begin{pmatrix}a &0\\0&1\end{pmatrix} = \begin{cases} \tau(a) |a| &\text{if } | a|_p \leq1, \\ 0&
\text{otherwise} \end {cases}
\end{equation}

\begin{equation}
W_p'\left(\begin{pmatrix}a &0\\0&1\end{pmatrix} \begin{pmatrix} 0 & 1 \\
-1& 0 \end{pmatrix} \right)=
\begin{cases} -p^{-1}\tau(a) |a| &\text{if } | a|_p \leq p, \\ 0&
\text{otherwise} \end {cases}
\end{equation}

We extend $W_p'$ to a function on $GU(1,1)(\Q_p)$ by
$$W_p'(ag)=W_p'(g), \text{ for } a\in L_p^\times, g\in GL_2(\Q_p).$$

Then, $W_p( \,s)$ is the unique function on $\G(\Q_p)$ such that
\begin{equation}
W_p(mnuk,s) = W_p(mu,s), \text{ for } m\in M(\Q_p), n\in N(\Q_P), u \in\{1, s_1\}, k\in I_p',
\end{equation}

\begin{equation}
W_p (t)=0 \text{ if } t \notin P(\Q_p)I_p' \sqcup P(\Q_p)s_1I_p'
\end{equation}
\begin{equation}\begin{split}\label{e:endf3}
W_p&\left(\begin{pmatrix} a&0&0&0\\0&1&0&0\\0&0&\overline{a}^{-1}&0\\0&0&0&1\end{pmatrix}\begin{pmatrix}
1&0&0&0\\0&a_1&0&b_1\\0&0&c_1&0\\0&d_1&0&e_1\end{pmatrix}u,s\right)
 \\&= \left|N_{L/\Q}(a)\cdot c_1^{-1} \right|_p^{3(s+1/2)}\cdot
 \Lambda_p(a)W_p'\begin{pmatrix}a_1&b_1\\d_1&e_1\end{pmatrix},
 \end{split}\end{equation}
for $a\in \Q_p^\times, u \in\{1, s_1\},
\begin{pmatrix}a_1&b_1\\d_1&e_1\end{pmatrix}\in GU(1,1)(\Q_p), c_1 =
\mu_1\begin{pmatrix}a_1&b_1\\d_1&e_1\end{pmatrix}$.

\subsection{The results}

We now state and prove the main theorem of this section.
\begin{theorem}\label{t:steinbergsteinberg}
Let the functions $B_p, W_p$ be as defined in subsection~\ref{s:bpwp2}. Then we have
 $$Z_p(s,W_p, B_p) =
 \frac{1-p}{p^2+1} \cdot \frac{p^{-6s-3}}{1-a_pw_p p^{-3s-3/2}} \cdot
L(3s+\frac{1}{2}, \pi_p \times \sigma_p)$$\\

where $L(s, \pi_p \times \sigma_p) = (1+ a_p w_p p^{-1}  p^{-s})^{-1}(1+ a_p w_p p^{-2}
p^{-s})^{-1}.$

\end{theorem}

Before we begin the proof, we need a lemma.

\begin{lemma}\label{l:steinbergramifiedwhittakerramified} We have the following formulae for $W_p(\Theta h(l,m)t_i,s)$ where $t_i \in T_m.$

$W_p(\Theta h(l,m)t_i,s) =\begin{cases} p^{-6ms-3ls-3m-5l/2}a_p^l\cdot \frac{-1}{p} & \text{if } i =3,4, \quad m>0\\ p^{-6ms-3ls-3m-5l/2}a_p^l & \text{if } i =5,6, \quad m>0 \\
0 & \text{otherwise.}\end{cases}$

\end{lemma}

\begin{proof}

We have \begin{equation}\label{e:thetausedlater}\Theta h(l,m) = h(l,m)\begin{pmatrix}1
&0&0&0\\p^m\alpha&1&0&0\\0&0&1&-p^m\overline{\alpha}\\0&0&0&1
\end{pmatrix}\end{equation}
 Put $K_p' = r^{-1}(G(F_p))$. Thus $\Theta h(l,m) t_i \in  P(\Q_p)K_p'$ when $m>0$ and $\Theta h(l,m) t_i \in
 P(\Q_p)\Theta K_p'$ when $m=0$. A direct computation shows that  $P(\Q_p)K_p'$ and $P(\Q_p)\Theta K_p'$ are disjoint; the fact that $P(\Q_p)I_p' \subset P(\Q_p)K_p'$ then implies that $W_p(\Theta h(l,m)t_i,s) = 0$ for $m=0$ . From now on we assume $m>0$.

We can check that $\Theta h(l,m)t_i \notin P(\Q_p)I_p'$ if $i \in \{1,2,7,8\}.$

For the remaining $t_i$ (i.e. $i \in \{3,4,5,6\}$) we have the decompositions:

$\Theta h(l,m)t_3$ = $$\begin{pmatrix}p^{2m+l} &0&0&0\\0&1&0&0\\0&0&p^{-2m-l}&0\\0&0&0&1
\end{pmatrix}m^{(2)}(\begin{pmatrix}p^{m+l}&0\\-p^m&p^m\end{pmatrix} )\begin{pmatrix}-1&0&0&0\\-p^m \alpha&-1&0&0\\0&-p^m \overline{\alpha}&-1&p^m
\overline{\alpha}\\-p^m \alpha&0&0&-1\end{pmatrix}$$

 $\Theta h(l,m)t_4$ = $$\begin{pmatrix}p^{2m+l} &0&0&0\\0&1&0&0\\0&0&p^{-2m-l}&0\\0&0&0&1
\end{pmatrix}m^{(2)}(\begin{pmatrix}p^{m+l}&0\\-p^m&p^m\end{pmatrix} )s_1\begin{pmatrix}1&p^m \alpha&0&0\\0&1&0&0\\0&p^m \alpha&1&0\\p^m \overline{\alpha}&0&-p^m \overline{\alpha}&1\end{pmatrix}$$

 $\Theta h(l,m)t_5$ $$= \begin{pmatrix}p^{2m+l}
&0&0&0\\0&1&0&0\\0&0&p^{-2m-l}&0\\0&0&0&1
\end{pmatrix}m^{(2)}(\begin{pmatrix}p^{m+l}&0\\0&p^m\end{pmatrix} )\begin{pmatrix}1&0&0&0\\p^m \alpha&1&0&0\\0&0&1&p^m
\overline{\alpha}\\0&0&0&1\end{pmatrix}$$

$\Theta h(l,m)t_6$ $$= \begin{pmatrix}p^{2m+l}
&0&0&0\\0&1&0&0\\0&0&p^{-2m-l}&0\\0&0&0&1
\end{pmatrix}m^{(2)}(\begin{pmatrix}p^{m+l}&0\\0&p^m\end{pmatrix} )s_1\begin{pmatrix}1&p^m \alpha&0&0\\0&1&0&0\\0&0&1&0\\0&0&-p^m
\overline{\alpha}&1\end{pmatrix}$$

The lemma now follows from the above decompositions and equations~\eqref{e:beginf3}-\eqref{e:endf3}.

\end{proof}

\begin{proof}[Proof of Theorem~\ref{t:steinbergsteinberg}]
By \eqref{e:keyformula}  we have

\begin{equation}\label{e:wp1t2}
Z_p(s,W_p,B_p) = \sum_{l\geq 0, m \ge 0}\sum_{t_i \in T_m}B_p( h(l,m)t_i) W_p(\Theta h(l,m)t_i,s)\cdot
I_{t_i}^{l,m}
\end{equation}

From Proposition~\ref{p:table}, Proposition~\ref{p:besselvalues1} and Lemma~\ref{l:steinbergramifiedwhittakerramified}  we have $$\sum_{i\in \{3,4,5,6\} } B_p( h(l,m)t_i) W_p(\Theta h(l,m)t_i,s)\cdot
I_{t_i}^{l,m} = \frac{(1-p)(-a_pw_pp^{-3s-5/2})^l(p^{-6s-3})^m}{p^2+1} .$$

Hence \eqref{e:wp1t2} implies $$Z_p(s,W_p,B_p) = \frac{(1-p)p^{-6s-3}}{p^2+1} \cdot\frac{1}{1+a_pw_pp^{-2}p^{-3s-1/2}}\cdot\frac{1}{1-p^{-6s-3}}$$

This completes the proof.

\end{proof}
\textbf{Remark.} We might equally well have chosen $W_p$ to be the simpler vector supported only on $1$ (rather than on 1 and $s_1$). Indeed, all the results in this paper will remain valid with that choice. The reason we include $s_1$ in the support of the section is because this definition will be necessary for our future work~\cite{sah2}.

\section{The case Steinberg $\pi_p$, unramified $\sigma_p$ }

\subsection{Assumptions}
Suppose that the characters $\omega_\pi , \omega_\sigma, \chi_0$ are trivial. Let $p \neq 2$ be a finite prime
of $\Q$ such that

\begin{enumerate}

\item $p$ is inert in $L =\Q(\sqrt{-d})$.

\item\label{i:lambda} $\Lambda_p$ is not trivial on $T(\Z_p)$; however it is trivial on $T(\Z_p)\cap \Gamma^0_p$.

\item $\pi_p$ is the Steinberg representation (or its twist by the unique non-trivial unramified quadratic character)
while $\sigma_p$ is unramified.

\item The conductor of  $\psi_p$ is
 $\Z_p.$

\item $S = \begin{pmatrix} a & b/2 \\ b/2 & c\\ \end{pmatrix} \in
M_2(\Z_p)$.

\item $-d = b^2 - 4ac$ generates the discriminant of $L_p/ \Q_p.$

\end{enumerate}

\textbf{Remark}. $\pi_p$ corresponds to a local newform for the Iwahori subgroup $I_p$ (see \cite{sch}).

\subsection{Description of $B_p$ and $W_p$}\label{s:bpwp}
Let $\Phi_p$ be the unique normalized local newform for the Iwahori subgroup $I_p$, as defined by
Schmidt~\cite{sch}. Let $w_p$ be the local Atkin-Lehner eigenvalue for $\pi_p$; this equals $-1$ when $\pi_p$ is
the Steinberg representation and equals $1$ when $\pi_p$ is the unramified quadratic twist of the Steinberg
representation. We let $B_p$ be the normalized vector that corresponds to $\Phi_p$ in the Bessel space.
Section~\ref{s:appendix} was devoted to the computation of the values $B_p(h(l,m)t)$ for $l, m \in \Z, m
\geq 0, t\in T_m.$\\

We now define $W_p$.
Take the canonical map $r: \widetilde{K}_p \rightarrow \widetilde{G}
(\F_p)$ and define
$I_p' = r^{-1}(I(\F_p))$, where $I(\F_p)$ is the subgroup of
$G(\F_p)$ defined in the beginning of this paper.

Let $W_p(\ ,s)$ be the unique vector in $I(\Pi_p,s)$ with the
 following properties:
 \begin{itemize}
 \item $W_p(\Theta,s)=1$,
 \item $W_p(1,s)=1$,
 \item $W_p(gk,s)=W_p(g,s)$ is $k \in I_p'$,
 \item $W_p(g,s)=0$ if $g$ does not belong to $P(\Q_p)\Theta I_p' \sqcup P(\Q_p) I_p'$
 \end{itemize}

Concretely we have the following description of $W_p( \ ,s).$

Suppose $\sigma_p$ is the principal series representation induced from the unramified characters $\alpha$,
$\beta$ of $\Q_p^\times$. Let $W_p'$ be the unique function on $GL_2(\Q_p)$ such that
\begin{equation}\label{e:beginf4}
W_p'(gk) = W_p'(g), \text{for } g \in GL_2(\Q_p), k\in  GL_2(\Z_p),
\end{equation}
\begin{equation}
W_p'(\begin{pmatrix}1 &x\\0&1\end{pmatrix}g) = \psi_p(-cx)W_p'(g), \text{for } g \in GL_2(\Q_p), x\in \Q_p,
\end{equation}
\begin{equation}
W_p'\begin{pmatrix}a &0\\0&b\end{pmatrix} = \begin{cases} \left| \frac{a}{b}\right|_p^\frac{1}{2} \cdot
\frac{\alpha(ap)\beta(b) -\alpha(b)\beta(ap)}{\alpha(p)-\beta(p)} &\text{if } \left| \frac{a}{b}\right|_p \leq1,
\\ 0& \text{otherwise} \end {cases}
\end{equation}

We extend $W_p'$ to a function on $GU(1,1)(\Q_p)$ by
$$W_p'(ag)=W_p'(g), \text{ for } a\in L_p^\times, g\in GL_2(\Q_p).$$

Then, $W_p( \,s)$ is the unique function on $\G(\Q_p)$ such that
\begin{equation}
W_p(mnu k,s) = W_p(mu,s), \text{ for } m\in M(\Q_p), n\in N(\Q_P), u\in \{1, \Theta \}, k\in I_p',
\end{equation}

\begin{equation}
W_p (t)=0 \text{ if } t \notin P(\Q_p) \Theta I_p' \sqcup P(\Q_p) I_p'
\end{equation}
\begin{equation}\begin{split}\label{e:endf4}
W_p&\left(\begin{pmatrix} a&0&0&0\\0&1&0&0\\0&0&\overline{a}^{-1}&0\\0&0&0&1\end{pmatrix}\begin{pmatrix}
1&0&0&0\\0&a_1&0&b_1\\0&0&c_1&0\\0&d_1&0&e_1\end{pmatrix}u,s\right)
 \\&=  \left|N_{L/\Q}(a)\cdot c_1^{-1} \right|_p^{3(s+1/2)}\cdot
 \Lambda_p(\overline{a}^{-1})W_p'\begin{pmatrix}a_1&b_1\\d_1&e_1\end{pmatrix},
 \end{split}\end{equation}
for $a\in \Q_p^\times, u\in \{1, \Theta \}, \begin{pmatrix}a_1&b_1\\d_1&e_1\end{pmatrix}\in GU(1,1)(\Q_p), c_1 =
\mu_1\begin{pmatrix}a_1&b_1\\d_1&e_1\end{pmatrix}$.

\subsection{The results}

We now state and prove the main theorem of this section.
\begin{theorem}\label{t:steinbergunramified}
Let the functions $B_p, W_p$ be as defined in subsection~\ref{s:bpwp}. Then we have
 $$Z_p(s,W_p, B_p) = \frac{1}{(p+1)(p^2+1)} \cdot
L(3s+\frac{1}{2}, \pi_p \times \sigma_p),$$\\

where $L(s, \pi_p \times \sigma_p) = (1+ w_p p^{-3/2} \alpha(p) p^{-s})^{-1}(1+ w_p p^{-3/2} \beta(p)
p^{-s})^{-1}.$

\end{theorem}

Before we begin the proof, we need a lemma.
\begin{lemma}\label{l:whittaker}
Let $t_i \in T_m, l\ge 0 $. We have

$W_p(\Theta h(l,m)t_i,s) =\begin{cases} p^{-3ls-2l}\left(\frac{\alpha(p)^{l+1} -
\beta(p)^{l+1}}{\alpha(p)-\beta(p)}\right) & \text{if } m =0,  i = 5  \\ p^{-6ms-3ls-3m-5l/2}\left(\frac{\alpha(p)^{l+1} -
\beta(p)^{l+1}}{\alpha(p)-\beta(p)}\right) & \text{if } m >0, i = 3,5\\
0 & \text{otherwise}\end{cases}$

\
\end{lemma}
\begin{proof}
By the proof of Lemma~\ref{l:steinbergramifiedwhittakerramified} we have
$\Theta h(l,m)t_i \notin P(\Q_p)\Theta I_p'$ if  $m > 0$. As for the case $m=0$, we can check that $\Theta h(l,0) t_i \notin P(\Q_p)\Theta I_p'$ if $i \in \{1,2,7\}$.
 On the other hand, again by the proof of Lemma~\ref{l:steinbergramifiedwhittakerramified}, we have  $\Theta h(l,m)t_i \in P(\Q_p) I_p'$ if and only if  $m > 0$ and $i \in \{3, 5\}$.
The lemma now follows immediately
 from~\eqref{e:beginf4}~-~\eqref{e:endf4}.
\end{proof}

\begin{proof}[Proof of Theorem~\ref{t:steinbergunramified}]

 We have
\begin{equation}\label{e:wp4t}\begin{split}
Z_p(s,W_p,B_p) &= \sum_{l\geq 0}W_p(\Theta h(l,0)t_5,s) B_p(h(l,0)t_5)\cdot
I_{t_5}^{l,0}\\ &+ \sum_{l\geq 0,m>0}\sum_{i \in \{3,5\}}W_p(\Theta h(l,m)t_i,s) B_p(h(l,m)t_i)\cdot
I_{t_i}^{l,m}\end{split}
\end{equation}

Using Proposition~\ref{p:besselvalues2} , Proposition~\ref{p:table2} and Lemma~\ref{l:whittaker} we have

$$\sum_{i \in \{3,5\}}W_p(\Theta h(l,m)t_i,s) B_p(h(l,m)t_i)\cdot
I_{t_i}^{l,m}=0$$ and hence

\begin{align*}Z_p(s,W_p,B_p)&=\sum_{l\geq 0}W_p(\Theta h(l,0)t_5,s) B_p(h(l,0)t_5)\cdot
I_{t_5}^{l,0}\\&=\frac{1}{(p+1)(p^2+1)}\sum_{l\geq 0} p^{-3ls-2l}\left(\frac{\alpha(p)^{l+1} -
\beta(p)^{l+1}}{\alpha(p)-\beta(p)}\right)(-pw_p)^l p^{-l}\\
&=\frac{1}{(p+1)(p^2+1)} L(3s+\frac{1}{2}, \pi_p \times \sigma_p).  \\
\end{align*}

This completes the proof of the theorem.

\end{proof}

\textbf{Remark.} We might equally well have chosen $W_p$ to be the simpler vector supported only on $\Theta$ (rather than on $\Theta$ and $1$). The only reason we include $1$ in the support of the section is because this definition will be necessary for our future work~\cite{sah2}.

\section{The global integral and some results}
\subsection{Classical Siegel modular forms and newforms for the minimal congruence subgroup }

For $M$ a positive integer define the following global parahoric subgroups.

\begin{align*}B(M) &:= Sp(4,\Z) \cap \begin{pmatrix}\Z& M\Z&\Z&\Z\\\Z& \Z&\Z&\Z\\M\Z& M\Z&\Z&\Z\\M\Z&M \Z&M\Z&\Z\\\end{pmatrix},\\
U_1(M) &:= Sp(4,\Z) \cap \begin{pmatrix}\Z& \Z&\Z&\Z\\\Z& \Z&\Z&\Z\\M\Z& M\Z&\Z&\Z\\M\Z&M \Z&\Z&\Z\\\end{pmatrix},\\
U_2(M) &:= Sp(4,\Z) \cap \begin{pmatrix}\Z& M\Z&\Z&\Z\\\Z& \Z&\Z&\Z\\\Z& M\Z&\Z&\Z\\M\Z&M \Z&M\Z&\Z\\\end{pmatrix},\\
U_0(M) &:= Sp(4,\Q) \cap \begin{pmatrix}\Z& M\Z&\Z&\Z\\\Z& \Z&\Z&M^{-1}\Z\\M\Z& M\Z&\Z&\Z\\M\Z&M
\Z&M\Z&\Z\\\end{pmatrix}.\\ \end{align*}

When $M=1$ each of the above groups is simply $Sp(4,\Z)$. For $M>1$, the groups are all distinct. If $\Gamma'$ is equal to one of the above groups, or (more generally) is any congruence subgroup, we define $S_k(\Gamma')$ to be
the space of Siegel cusp forms of degree 2 and weight $k$ with respect to the group $\Gamma'$.

More precisely, let $\H_2= \{ Z \in M_2(\C) | Z =Z^T,i( \overline{Z} - Z)$ is positive definite$\}$. For any $g=
\begin{pmatrix} A&B\\ C&D \end{pmatrix} \in G$ let $J(g,Z) = CZ + D$. Then $f \in S_k(\Gamma')$ if it is a holomorphic
function on $\H_2$, satisfies $f(\gamma Z) = \det(J(\gamma,Z))^k f(Z)$ for $\gamma \in \Gamma',Z \in \H_2$ and
vanishes at the cusps. It is well-known that $f$ has a Fourier expansion $$f(Z) =
\sum_{S > 0} a(S, F) e(\text{tr}(SZ)),$$ where $e(z) = \exp(2\pi iz)$ and $S$ runs through all symmetric
semi-integral positive-definite matrices of size two.

Now let $M$ be a square-free positive integer. For any decomposition $M= M_1M_2$ into coprime integers we
define, following Schmidt~\cite{sch}, the subspace of oldforms $S_k(B(M))^{\text{old}}$ to be the sum of the
spaces

$$S_k(B(M_1) \cap U_0(M_2)) + S_k(B(M_1) \cap U_1(M_2)) + S_k(B(M_1) \cap U_2(M_2)).$$

For each prime $p$ not dividing $M$ there is the local Hecke algebra $\mathfrak{H}_p$ of operators on
$S_k(B(M))$ and for each prime $q$ dividing $M$ we have the Atkin-Lehner involution $\eta_q$ also acting on
$S_k(B(M))$. For details, the reader may refer to \cite{sch}.

By a newform for the minimal congruence subgroup $B(M)$, we mean an element $f \in  S_k(B(M))$ with the
following properties
\begin{enumerate}
\item \label{itemfa} $f$ lies in the orthogonal complement of the space $S_k(B(M))^{\text{old}}.$
\item \label{itemfb}$f$ is an eigenform for the local Hecke algebras $\mathfrak{H}_p$ for all primes $p$ not dividing $M$.
\item \label{itemfc}$f$ is an eigenform for the Atkin-Lehner involutions $\eta_q$ for all primes $q$ dividing $M$.
\end{enumerate}

\textbf{Remark.} By \cite{sch}, if we assume the hypothesis that a nice $L$-function theory for
$GSp(4)$ exists, \eqref{itemfb} and \eqref{itemfc} above follow from \eqref{itemfa} and the assumption that $f$
is an eigenform for the local Hecke algebras at \emph{almost} all primes.

\subsection{Description of our newforms}\label{s:newformdef}
 Let $M$ be an odd square-free positive integer and $$F(Z) = \sum_{T>0} a(T)\text{e}(\text{tr}(TZ)) $$ be a Siegel newform for $B(M)$ of even weight $l$.

Let $N$ be an odd square-free positive integer and $ g$ be a normalized newform of weight $l$ for $\Gamma_0(N)$. $ g$ has a Fourier expansion $$ g(z) = \sum_{n=1}^\infty b(n)e(nz)$$ with $b(1) =1$. It is then well known that the $b(n)$ are all totally real algebraic numbers.

We make the following assumption:

\begin{equation}\label{fundamentalrestriction}a(T) \neq 0
\text{ \emph{for some} } T
= \begin{pmatrix} a& \frac{b}{2} \\ \frac{b}{2} & c \end{pmatrix}\end{equation} \emph{such that} $-d = b^2
-4ac$  \emph{is the discriminant of the imaginary quadratic field} $\Q(\sqrt{-d}),$
 \emph{and all primes dividing} $MN$ \emph{are inert in}  $\Q(\sqrt{-d}).$

We define a function $\Phi = \Phi_{F}$ on $G(\A)$ by
$$\Phi(\gamma h_\infty k_0) = \mu_2(h_\infty)^l\det(J(h_\infty, iI_2))^{-l}F(h_\infty(i))$$
where $\gamma \in G(\Q), h_\infty \in G(\R)^+$ and $$k_0 \in (\prod_{ p \nmid M} K_p) \cdot (\prod_{p \mid M}I_p). $$

Because we do not have strong multiplicity one for $G$ we can only say that the representation of $G(\A)$
generated by $\Phi$ is a \emph{multiple} of an irreducible representation $\pi$. However that is enough for our
purposes.

We know that $\pi =\otimes \pi_v$ where $$\pi_v =  \begin{cases}\text{holomorphic discrete series} & \text{ if } v=\infty,\\
 \text{unramified spherical principal series} &\text{ if } v \text{ finite }, v \nmid M, \\
 \xi_v \text{St}_{GSp(4)}  \text{where } \xi_v \text{ unramified, } \xi_v^2=1 &\text{ if } v \mid M. \end{cases}$$

Next, we define a function $\Psi$ on $GL_2(\A)$ by $$\Psi(\gamma_0mk_0) = (\det \
m)^{\frac{l}{2}}(\gamma i + \delta)^{-l} g(m(i))$$ where $\gamma_0 \in GL_2(\Q), \ m = \begin{pmatrix}\alpha
&\beta\\ \gamma & \delta \end{pmatrix} \in GL_2^+(\R),$ and $$k_0 \in \prod_{p \nmid
N}GL_2(\Z_p)\prod_{p\mid N}\Gamma_{0,p} $$ Let $\sigma$ be the automorphic representation of
$GL_2(\A)$ generated by $\Psi$.

We know that $\sigma =\otimes \sigma_v$ where $$\sigma_v =  \begin{cases}\text{holomorphic discrete series} & \text{ if } v=\infty,\\
 \text{unramified spherical principal series} &\text{ if } v \text{ finite }, v \nmid N, \\
 \xi \text{St}_{GL(2)}  \text{where } \xi_v \text{ unramified, } \xi_v^2=1 &\text{ if } v\mid N. \end{cases}$$

\subsection{Description of our Bessel model}
In order to use our results from the previous sections, we need to associate a Bessel model to $\pi$ (or more
accurately, we associate it to $\overline{\pi}$). This involves making a choice of $(S,\Lambda,\psi)$. This subsection is devoted to doing that.

Let $\psi =  \prod_v\psi_v$ be a character of $\A$ such that
\begin{itemize}
\item The conductor of $\psi_p$ is $\Z_p$ for all (finite) primes $p$,
\item $\psi_\infty(x) = e(-x),$ for $x \in \R$,
\item $\psi|_\Q =1.$
\end{itemize}

Put $L=\Q(\sqrt{-d}).$ where $d$ is the integer defined in~\eqref{fundamentalrestriction}.

 First we deal with the case $M=1$. In this case, our choice of $S$ and $\Lambda$ is identical to \cite{fur}. To recall, put \begin{equation}\label{e:tcosetca1}T(\A) = \coprod_{j=1}^{h(-d)}t_jT(\Q)T(\R)(\Pi_{p<\infty}
T(\Z_p))\end{equation} where $t_j \in \prod_{p<\infty} T(\Q_p)$ and $h(-d)$ is the class number of $L$.

Write $t_j = \gamma_{j}m_{j}\kappa_{j}$, where $\gamma_{j} \in GL_2(\Q), m_{j} \in GL_2^+(\R)$,
and $\kappa_{j}\in ((\Pi_{p<\infty} GL_2(\Z_p)).$

Choose $$S = \begin{cases}\begin{pmatrix}d/4&0\\0&1\end{pmatrix} & \text{ if } d\equiv 0 \pmod{4} \\ \begin{pmatrix}(1+d)/4&1/2\\1/2&1\end{pmatrix}&\text{ if }d\equiv3\pmod{4}\end{cases}$$

Let $S_{j} = \det(\gamma_{j})^{-1}\gamma_{j}^TS\gamma_{j}$. Then, any primitive semi-integral two by two positive definite matrix with
discriminant equal to $-d$ is $SL_2(\Z)$-equivalent to some $S_{j}$.
 So, by our assumption, we can choose $\Lambda$ a character of $T(\A)/T(\Q)T(\R)((\Pi_{p<\infty} T(\Z_p))$ such that $$\sum_{j=1}^{h(-d)}\Lambda(t_j)\overline{a(S_j)}\ne 0.$$
Thus, we have specified a choice of $S$ and $\Lambda$ for $M=1$.

\emph{In the rest of this subsection, unless otherwise mentioned, assume $M>1$.}

Suppose $p$ is a prime dividing $M$. We can identify $L_p$ with elements $a + b\sqrt{-d}$ with $a,b \in \Q_p$. Let
$\Z_{L,p}^\times$ denote the units in the ring of integers of $L_p$. The elements of $\Z_{L,p}^\times$ are of
the form  $a + b\sqrt{-d}$ with $a,b \in \Z_p$ and such that at least one of $a$ and $b$ is a unit. Let
$\Gamma_{L,p}^0$ be the subgroup of $\Z_{L,p}^\times$ consisting of the elements with $p | b$. The group
$\Z_{L,p}^\times / \Gamma_{L,p}^0$ is clearly cyclic of order $p + 1$. Moreover, the elements $\{ (-b +
\sqrt{-d})/2 \}$ where $b$ is a positive integer satisfying $\{1 \leq b \leq 2p : b = d \pmod{2} \}$ are
distinct in   $\Z_{L,p}^\times / \Gamma_{L,p}^0$. Note that $d = 0 \text{ or } 3 \pmod{4}$ and hence $b = d
\pmod{2}$ implies that $4$ divides $b^2 + d$.  So we have the lemma:
\begin{lemma} There exists an integer $b$ such that $4$ divides $b^2 + d$ and $(-b + \sqrt{-d})/2$ is a generator of the group $\Z_{L,p}^\times / \Gamma_{L,p}^0$ for each $p | M$.
\end{lemma}
\begin{proof}
By the comments above, we can choose, for each prime $p_i$ dividing $M$, an integer $b_i$ such that $b_i \equiv d \pmod{2}$ and $(-b_i + \sqrt{-d})/2$ is a generator of the group $\Z_{L,p_i}^\times / \Gamma_{L,p_i}^0$. Now, using the Chinese Remainder theorem, choose $b$ satisfying $b \equiv b_i \pmod{2p_i}$ for each $i$ .
\end{proof}

Now we define $$S = \begin{pmatrix}\frac{b^2 +d}{4}&\frac{b}{2}\\ \frac{b}{2} &1 \end{pmatrix}.$$

As in section~\ref{s:bessel} we define the matrix $\xi = \xi_S$ and the group $T = T_S$. We have $T(\Q) \simeq
L^\times.$ We write $T(\Z_p)$ for $T(\Q_p)\cap GL_2(\Z_p).$

Let \begin{equation}\label{e:tcoset}T(\A) = \coprod_{j=1}^{h(-d)}t_jT(\Q)T(\R)(\Pi_{p<\infty}
T(\Z_p)\end{equation} where $t_j \in \prod_{p<\infty} T(\Q_P)$ and $h(-d)$ is the class number of $L$. For each $p |M$ put
$\Gamma_{L,p}^0 = T(\Z_p) \cap \Gamma^0_p$. Note that under the isomorphism $T(\Z_p) \simeq \Z_{L,p}^\times$
sending $x +y \xi \mapsto x +y\frac{\sqrt{-d}}{2}$, our two definitions for $\Gamma_{L,p}^0$ agree, so there is no ambiguity.

Let $M= p_1p_2...p_r$ be its decomposition into distinct primes. For each $ 1\le i \le r$ we choose coset representatives $u_{k_i}^{(p_i)} \in T(\Z_{p_i})$ such that $$T(\Z_{p_i})
= \coprod_{k_i=1}^{p_i+1}u_{k_i}^{(p_i)}\Gamma_{L,p_i}^0.$$

We write an $r$-tuple $(k_1,..,k_r)$ in short as $\widetilde{k}$. Let $X$ denote the cartesian product of the $r$ sets $X_i = \{x : 1 \le x \le p_i \}$. For $\widetilde{k} \in X$, define $$u_{\widetilde{k}} = \prod_{i=1}^r u_{k_i}^{(p_i)}.$$

Then it is easy to see that as $\tk$ varies over $X$ the elements $u_{\tk}$ form a set of coset representatives of $\Pi_{p | M}T(\Z_p) / \Pi_{p|M}\Gamma^0_{L,p}.$ Also note that $|X| =  |SL_2(\Z)/\Gamma^0(M)|=\Pi_{p_1 |M}(p_i + 1).$ We denote the quantity $$\prod_{p_1 |M}(p_i + 1)$$ by $g(M).$

Let $T(\Z)$ denote the (finite) group of units in the ring of integers $\Z_L$ of $L$. Let $t(d)$ denote the cardinality of the group $T(\Z)/\{\pm1\}$. We know that,
$$t(d) =\begin{cases} 3 &\text{ if } d=3\\ 2 &\text{ if } d=4\\1 &\text{ otherwise. } \end{cases}$$
Let $T_M^\times$ be the image of $T(\Z)$ in $\Pi_{p | M}T(\Z_p)$. Then  $T_M^\times \cap \Pi_{p|M}\Gamma^0_{L,p} = \{\pm1\}.$ Choose a set of elements $r_1, r_2,..r_{t(d)}$ in $T(\Z)$ such that they form distinct representatives in $T(\Z)/\{\pm1\}$. Let $\tr_i$ denote the image of $r_i$ in $T_M^\times$. We have \begin{equation}\label{tdeco1}T_M^\times\Pi_{p|M}\Gamma^0_{L,p}= \coprod_{i=1}^{t(d)} \tr_i (\Pi_{p|M}\Gamma^0_{L,p}).\end{equation}

Finally, choose $x_1,x_2,...,x_{g(M)/t(d)}$ in $\Pi_{p | M}T(\Z_p)$ such that we have the disjoint coset decomposition:
\begin{equation}\label{tdeco2}\Pi_{p | M}T(\Z_p)= \coprod_{i=1}^{g(M)/t(d)}x_iT_M^\times \Pi_{p|M}\Gamma^0_{L,p}\end{equation}

This immediately gives us the fundamental coset decomposition:
\begin{equation}\label{e:tcosetdec}T(\A)
= \coprod_{\substack{1\leq j \leq h(-d)\\ 1 \le k \le g(M)/t(d)}}t_jx_kT(\Q)T(\R)(\Pi_{p \nmid M}
T(\Z_p))(\Pi_{p_i | M} \Gamma_{L,p_i}^0)\end{equation}

Also from~\eqref{tdeco1} and~\eqref{tdeco2} we immediately get another coset decomposition:

\begin{equation}\label{tdeco5}\Pi_{p | M}T(\Z_p)= \coprod_{\substack{1\le i \le g(M)/t(d)\\1\le j \le t(d)} }x_i\tr_j \Pi_{p|M}\Gamma^0_{L,p}\end{equation}

But we know that an alternate set of coset representatives in the above equation is given by the elements $u_{\tk}$. It follows that for any $1\le i \le g(M)/t(d),1\le j \le t(d)$, there exists a unique $\tk \in X$ such that $u_{\tk}^{-1}x_i\tr_j \in \Pi_{p|M}\Gamma^0_{L,p}$. This correspondence is bijective.

Write $t_j x_k = \gamma_{j,k}m_{j,k}\kappa_{j,k}$, where $\gamma_{j,k} \in GL_2(\Q), m_{j,k} \in GL_2^+(\R)$,
and $\kappa_{j,k}\in (\Pi_{p<\infty, p \nmid M} GL_2(\Z_p)\cdot \Pi_{p|M}\Gamma^0_p .$ Also, by $(\gamma_{j,k})_f$ we
denote the finite part of $\gamma_{j,k}$, that is, $(\gamma_{j,k})_f=\gamma_{j,k}m_{j,k}.$

\begin{lemma}\label{l:gammarep} For each $j$, the elements $\gamma_{j,1}^{-1}r_l\gamma_{j,k}$ form a system of representatives of $SL_2(\Z) / \Gamma^0(M)$ as $l,k$ vary over $1 \le l \le t(d)$, $1 \le k \le g(M)/t(d)$.
\end{lemma}
\begin{proof}Fix $j$. Let $1 \le l_2 \le t(d)$, $1 \le k_2 \le g(M)/t(d)$.
We have $$\gamma_{j,k_2}^{-1}r_{l_2}^{-1}r_{l}\gamma_{j,k} = m_{j,k_2}\kappa_{j,k_2}x_{k_2}^{-1}r_{l_2}^{-1}r_lx_k(m_{j,k}\kappa_{j,k})^{-1}.$$
Therefore $\gamma_{j,k_2}^{-1}r_{l_2}^{-1}r_{l}\gamma_{j,k} \in \left(GL_2^+(\R)\Pi_{q<\infty}GL_2(\Z_q)\right)\cap GL_2(\Q) =
SL_2(\Z).$ Moreover, if it belongs to $\Gamma^0(M)$ then we must have $x_{k_2}^{-1}\tr_{l_2}^{-1}\tr_lx_k \in \Pi_{p|M}\Gamma^0_p$ and by ~\eqref{tdeco5} this can happen only if $l=l_2, k=k_2$. Now the lemma follows because the size of the set $\gamma_{j,1}^{-1}r_l\gamma_{j,k}$ equals the cardinality of $SL_2(Z)/\Gamma^0(M)$.
\end{proof}

Let $S_{j,k} = \det(\gamma_{j,k})^{-1}\gamma_{j,k}^TS\gamma_{j,k}$. So, looking at $S$ and $S_{j,k}$ as elements of $GL_2(R)^+$ we have $S_{j,k}= \det(m_{j,k})\
(m_{j,k}^{-1})^TSm_{j,k}^{-1}$.

\begin{lemma}\label{Sjknonzero}There exists $j,k$, $1\le j \le h(-d)$, $1\le k \le g(M)/t(d)$ such that $a(S_{j,k})\ne0$.
\end{lemma}
\begin{proof} By assumption~\eqref{fundamentalrestriction}, $a(T)\ne 0$ for some primitive semi-integral positive definite matrix $T$ with discriminant equal to $-d$. By~\cite[p.209]{fur} there exists $j$ such that $T$ is $SL_2(\Z)$-equivalent to $S_{j,1}$. This means there is $R\in SL_2(\Z)$ such that $T= R^TS_{j,1}R$. By Lemma~\ref{l:gammarep}, we can find $k,l$ such that $R=\gamma_{j,1}^{-1}r_l\gamma_{j,k}g$ where $g \in \Gamma^0(M)$. This gives us \begin{align*}T &=g^T\gamma_{j,k}^Tr_l^T(\gamma_{j,1}^{-1})^T S_{j,1}\gamma_{j,1}^{-1}r_l\gamma_{j,k}g\\&=\det(\gamma_{j,k})^{-1} g^T\gamma_{j,k}^Tr_l^TSr_l\gamma_{j,k}g\\&=\det(\gamma_{j,k})^{-1} g^T\gamma_{j,k}^TS\gamma_{j,k}g\\&=g^TS_{j,k}g\end{align*}
Hence $0 \ne a(T) = a(g^TS_{j,k}g) = a(S_{j,k}),$ using the fact that the image of $g^T$ in $Sp_4(\Z)$ falls in $B(M)$ and $F$ is a modular form for $B(M)$.

\end{proof}

\begin{proposition}\label{p:lambdaa}
There exists a character $\Lambda$ of $T(\A)/(T(\Q)T(\R)\Pi_{p<\infty, p \nmid M} T(\Z_p)\cdot \Pi_{p \mid M}\Gamma_{L,p}^0)$
such that $$\sum_{\substack{1\leq j\leq h(-d) \\ 1 \leq k \leq g(M)/t(d)}}\Lambda(t_jx_k)^{-1}\overline{a(S_{j,k})} \neq 0.$$
Moreover for any such $\Lambda$ we have $\Lambda_p$ non-trivial on $T(\Z_p)$ for each prime $p |M$.

\end{proposition}
\begin{proof}
By Lemma~\ref{Sjknonzero} we can find $S_{j,k}$ such that $a(S_{j,k})\ne 0$. Hence using~\eqref{e:tcosetdec} we know that a character $\Lambda$ satisfying the condition listed in the proposition exists.

Let $\Lambda$ be such a character and $p_i$ a fixed prime dividing $M$. We will show that $\Lambda_{p_i} $ is not the trivial character on $T(\Z_{p_i})$.

For any $1\le j \le h(-d)$ and $\tk \in X$ we can write $t_j u_{\tk} = \gamma_{j,\tk}m_{j,\tk}\kappa_{j,\tk}$, where $\gamma_{j,\tk} \in GL_2(\Q), m_{j,\tk} \in GL_2^+(\R)$ and $\kappa_{j,k}\in (\Pi_{p<\infty, p \nmid M} GL_2(\Z_p)\cdot \Pi_{p|M}\Gamma^0_p .$

We put $S_{j,\tk} = \det(\gamma_{j,\tk})^{-1}\gamma_{j,\tk}^TS\gamma_{j,k}$

Suppose $\Lambda_{p_i}$ is trivial on $T(\Z_{p_i})$. We claim that \begin{equation}\label{propsumt}\sum_{\substack{1\leq j\leq h(-d) \\ \tk \in X}}\Lambda(t_ju_{\tk})^{-1}\overline{a(S_{j,\tk})} = 0.\end{equation}

Suppose we fix $k_1,k_2,..,k_{i-1},k_{i+1},..k_r$. For $1\le y\le p_i +1$, let $\tk^{y} \in X$ be the $r$-tuple obtained by putting $k_i=y.$ Then, by essentially the same argument as in Lemma~\ref{l:gammarep} we see that $\gamma_{j,\tk^{1}}^{-1}\gamma_{j,\tk^{y}}$ form a set of representatives of $\Gamma^0(M/p_i)/\Gamma^0(M)$. In particular, this implies, by~\cite[3.3.3]{sch}, that $\sum_y a(S_{j,\tk^{y}})=0$, and therefore, because  $\Lambda_{p_i}$ is trivial on $T(\Z_{p_i})$, we must have $\sum_y\Lambda(t_ju_{\tk^y})^{-1}a(S_{j,\tk^{y}})=0$.  It follows, by breaking up $$\sum_{\substack{1\leq j\leq h(-d) \\ \tk \in X}}\Lambda(t_ju_{\tk})^{-1}\overline{a(S_{j,\tk})}$$into quantities as above,~\eqref{propsumt} follows.

Given $1\le k \le g(M)/t(d),1\le l \le t(d)$, let $\tk(k,l)$ be the unique element in $X$ such that \begin{equation}\label{onemore}u_{\tk(k,l)}^{-1}x_k\tr_l \in \Pi_{p|M}\Gamma^0_{L,p}\end{equation}. Such an element exists by our comment after~\eqref{tdeco5}. Suppose we write $r_l = \tr_l r_{l,f} r_{l,\infty}$ where $r_{l,f} \in \Pi_{p \nmid M}T(\Z_p)$ and $r_{l,\infty}\in T(\R)$

Then, using~\eqref{onemore} we have $$t_ju_{\tk(k,l)}= r_lt_jx_kr_{l,\infty}^{-1}k$$ with $k \in (\Pi_{p<\infty, p \nmid M} GL_2(\Z_p)\cdot \Pi_{p|M}\Gamma^0_p$. In other words we can take $\gamma_{j,\tk(k,l)} = r_l\gamma_{j,k}$.

But then $a(S_{j,\tk(k,l)})= a(S_{j,k})$. Also from~\eqref{onemore} it is clear that $\Lambda^{-1}(t_ju_{\tk(k,l)})=\Lambda^{-1}(t_jx_k)$. On the other hand if we let $k,l$ vary over all elements in the range $1\le k \le g(M)/t(d),1\le l \le t(d)$, the corresponding $\tk(k,l)$ vary over all $\tk \in X$. As a result we conclude that \begin{equation}\sum_{\substack{1\leq j\leq h(-d) \\ \tk \in X}}\Lambda(t_ju_{\tk})^{-1}\overline{a(S_{j,\tk})} =t(d)\sum_{\substack{1\leq j\leq h(-d) \\ 1 \leq k \leq g(M)/t(d)}}\Lambda(t_jx_k)^{-1}\overline{a(S_{j,k})}\end{equation} But we have already shown that if $\Lambda_{p_i}$ is trivial on $T(\Z_{p_i})$ then $$\sum_{\substack{1\leq j\leq h(-d) \\ \tk \in X}}\Lambda(t_ju_{\tk})^{-1}\overline{a(S_{j,\tk})}=0.$$ The proof follows.

\end{proof}

From now on, fix a character $\Lambda$ satisfying the above proposition. Consider \begin{equation}\label{e:bdefinition}B_{\overline{\Phi}}(h) = \int_{Z_G(\A)R(\Q)\backslash
R(\A)}(\Lambda \otimes \theta)(r)^{-1}\overline{\Phi}(rh)dr\end{equation} where $\theta$ is defined as in
Section~\ref{prelim} and $\overline{\Phi}(h) = \overline{\Phi(h)}$.

Also, define $a(\Lambda) = a(F, \Lambda)$ by $$a(\Lambda) =\begin{cases}\sum_{1\leq j\leq h(-d)}\Lambda(t_j)a(S_{j})& \text{ if } M=1 \\\frac{1}{g(M)}\sum_{\substack{1\leq j\leq h(-d) \\ 1 \leq k \leq g(M)/t(d)}}\Lambda(t_jx_k)^{-1}\overline{a(S_{j,k})} & \text{ if }M>1. \end{cases}$$

\begin{proposition}\label{p:bproposi}Let $g_\infty \in G(\R)^+$. We have
$$B_{\overline{\Phi}}(g_\infty)=\overline{\det(J(g_\infty,i))^{-l}}\mu_2(g_\infty)^le(-\text{tr}(S\cdot \overline{g_\infty (i)})\sum_{1\leq j\leq h(-d) }a(F, \Lambda).$$

\end{proposition}

\textbf{Remark}. This is proved by explicit computation. The details for the case $M=1$ is there in~\cite[(1-26)]{sug}. The proof for the general case is completely analogous and hence not included here. The reader who wishes to see the details can take a look at the longer version of this paper available online~\cite{lflong}.

\subsection{Description of the Eisenstein series}\label{s:eisensteindef}
This section describes the Eisenstein series on $\G(\A)$. For each finite place $v$, recall that $\widetilde{K_v}$ is the maximal compact subgroup of $\G(\Q_v)$ and is defined by $$\widetilde{K_v}=\G(\Q_v) \cap GL_4(\Z_{L,v}).$$

Let us now define $$\widetilde{K_\infty} = \{ g
\in \G(\R) | \mu_2(g) =1, g<iI_2> = i I_2 \}.$$

Equivalently $$\widetilde{K_\infty} = U(2,2;\R) \cap U(4,\R).$$

We define $$\rho_l(k_\infty) = \det(k_\infty)^{l/2}\det(J(k_\infty,i))^{-l}.$$

By \cite[p. 5]{ich}, any matrix $k_\infty$ in $\widetilde{K_\infty}$ can be written in the form $k_\infty= \lambda\begin{pmatrix}A&B\\-B&A\end{pmatrix}$ where $\lambda\texttt{} \in  \C, |\lambda|=1$, and $A+iB,A-iB \in U(2;\R)$ with $\det(A+iB) = \overline{\det(A-iB)}.$ Then, \begin{equation}\label{e:defrhol}\rho_l(k_\infty)= \det(A-iB)^{-l}\end{equation}
Note that if $k_\infty$ has all real entries, i.e. $k_\infty \in \mathrm{Sp}(4,\R) \cap \mathrm{O}(4,\R)$, then $$\rho_l(k_\infty)=\det(J(k_\infty,i))^{-l}.$$

Extend $\Psi$ to $GU(1,1;L)(\A)$ by $$\Psi(ag) = \Psi(g)$$ for $a \in
L^\times(\A), g \in GL_2(\A).$
Now define the compact open subgroup $K^{\G}$ of $\G(\A_f)$ by
$$K^{\G}=\prod_{p<\infty, p\nmid MN}\widetilde{K_p}\prod_{ p\mid N, p\nmid M}\widetilde{U_r}\prod_{ p\mid M} I_p'$$

   Define \begin{equation}\label{defflambda}f_\Lambda(g,s) = \delta_P^{s+\frac{1}{2}}(m_1m_2)\Lambda(\overline{m_1})^{-1}
\Psi(m_2)
\rho_l(k_\infty) \quad \text{if } g=m_1m_2n\widetilde{k}k\end{equation} where
 $m_i \in M^{(i)}(\A) \quad (i=1,2)$, $n\in N(\A),$  $
 k=k_\infty k_0$ with  $k_\infty \in \widetilde{K_\infty}$, $k_0
\in K^{\G} $ and $\widetilde{k} = \prod_{p | M} k_p$ is such that $k_p \in \{1,s_1\}$ for $p | \gcd(M, N)$ and $k_p \in \{1,\Theta \}$ for $p | M, p \nmid N,$ and put $$ f_\Lambda(g,s) = 0$$ otherwise.

Finally, we define the Eisenstein series $E_{\Psi,\Lambda}(g,s)$ on $\widetilde{G}(\A)$ by
\begin{equation}E_{\Psi,\Lambda}(g,s) = \sum_{\gamma \in P(\Q) \backslash \G(\Q)}f_\Lambda(\gamma g,s).
\end{equation}

\subsection{The global integral}\label{ss:global}
The global integral for our consideration is $$Z(s) = \int_{Z_G(\A)G(\Q)\backslash
G(\A)}E_{\Psi,\Lambda}(g,s)\overline{\Phi}(g)dg.$$ Then, by~\eqref{e:local}, Theorem~\ref{t:furusawa}, Theorem~\ref{t:unramifiedsteinberg}, Theorem~\ref{t:steinbergsteinberg} and
Theorem~\ref{t:steinbergunramified} we have
\begin{equation}\label{e:globalintegralfinal}Z(s)= \frac{Q_fZ_\infty(s)}{g(M/f)P_{MN}}\cdot\prod_{p |f}
\frac{p^{-6s-3}}{1-a_{p}w_{p}p^{-3s-3/2}}\cdot \frac{L(3s + \frac{1}{2}, \pi \times \sigma)}{\zeta_{MN}(6s
+1)L(3s+1,\sigma \times \rho(\Lambda))} \end{equation} where $f$ denotes $\gcd(M,N)$ and $$L(s, \pi \times \sigma) = \prod_{q < \infty}L(s,
\pi_q \times \sigma_q)$$ $$L(s, \sigma \times \rho(\Lambda)))=\prod_{q<\infty, q \nmid M}L(s, \sigma_q \times
\rho(\Lambda_q)),$$
$$\zeta_{A}(s) = \prod_{\substack{p\nmid A\\ p \text{ prime }}} (1 - p^{-s})^{-1},$$ $$P_{A} = \prod_{\substack{r \mid A\\ r \text{ prime}}} (r^2 +1),$$, $$Q_{A} = \prod_{\substack{r \mid A\\ r \text{ prime}}} (1-r),$$  and \begin{equation}Z_\infty(s)=\int_{R(\R)\backslash G(\R)}W_{f_\Lambda}(\Theta g,s)B_{\overline{\Phi}}(g)dg \end{equation}

As for the explicit computation of $Z_\infty$, Furusawa's calculation in \cite{fur}, \textsl{mutatis mutandis},
works for us. The only real point of difference is the choice of $S$. Furusawa chooses $$S =
\begin{cases}\begin{pmatrix}\frac{d}{4}&0\\ 0 &1 \end{pmatrix}, & \text{ if }d \equiv 0
\pmod{4},\\\begin{pmatrix}\frac{1 +d}{4}&\frac{1}{2}\\ \frac{1}{2} &1 \end{pmatrix}, & \text{ if }d \equiv 3
\pmod{4}.\end{cases}$$ He computes $Z_\infty(s)$ for the case $d \equiv 0 \pmod{4}$ and uses it to deduce the
other case via a simple change of variables, using $$\begin{pmatrix}\frac{1 +d}{4}&\frac{1}{2}\\ \frac{1}{2} &1
\end{pmatrix}= \begin{pmatrix}1&0\\-\frac{1}{2}&1\end{pmatrix}^T \begin{pmatrix}\frac{d}{4}&0\\ 0 &1
\end{pmatrix} \begin{pmatrix}1&0\\-\frac{1}{2}&1\end{pmatrix}.$$

In our case we have, $$S=\begin{pmatrix}\frac{b^2 +d}{4}&\frac{b}{2}\\ \frac{b}{2} &1 \end{pmatrix} =
\begin{pmatrix}1&0\\\frac{b}{2}&1\end{pmatrix}^T \begin{pmatrix}\frac{d}{4}&0\\ 0 &1 \end{pmatrix}
\begin{pmatrix}1&0\\\frac{b}{2}&1\end{pmatrix}$$ and so a similar change of variables works.

Thus, we have (cf. \cite[p. 214]{fur}) $$Z_\infty(s) = \pi
\overline{a(\Lambda)}(4\pi)^{-3s - \frac{3}{2}l + \frac{3}{2}}d^{-3s-\frac{l}{2}}\cdot \frac{\Gamma(3s +
\frac{3}{2}l - \frac{3}{2})}{6s +l -1}.$$

Henceforth we simply write $L(s, F \times  g)$ for $L(s, \pi \times \sigma)$. We can summarize our
computations in the following theorem.

\begin{theorem}[The integral representation]\label{t:globalmain}Let $F$ and $E_{\Psi,\Lambda}$ be as defined previously. Then

$$\int_{Z_G(\A)G(\Q)\backslash G(\A)}E_{\Psi,\Lambda}(g,s)\overline{\Phi}(g)dg=C(s)\cdot L(3s + \frac{1}{2}, F \times  g)$$ where $C(s)=$ $$\frac{A(f)\pi \overline{a(\Lambda)}(4\pi)^{-3s - \frac{3}{2}l +
\frac{3}{2}}d^{-3s-\frac{l}{2}}\Gamma(3s +  \frac{3}{2}l - \frac{3}{2})}{g(M/f)P_{MN}(6s +l -1)\zeta_{MN}(6s+1)L(3s+1,\sigma \times \rho(\Lambda))} \prod_{p |d}
\frac{p^{-6s-3}}{1-a_{p}w_{p}p^{-3s-3/2}}$$ with $f = \gcd(M, N)$.
\end{theorem}

\textbf{Remark.} Note that $$C(\frac{l}{6}-\frac{1}{2}) = \frac{\pi^{4-2l} \quad \overline{a(F, \Lambda)} }{\zeta(l-2)L(\frac{l-1}{2},\sigma \times \rho(\Lambda))} \times \text{ (an algebraic number)}.$$

\section{A classical reformulation and special value consequences}
Let $$\G^+(\R) = \{g \in \G(\R) : \mu_2(g)>0\},$$ $$G^+(\R) = \{g \in G(\R) : \mu_2(g)>0\}.$$

Also, define $$\Ht_2= \{ Z \in M_4(\C) | i( \overline{Z} - Z) \text{ is positive definite}\}.$$ Note that $\G^+(\R)$ acts transitively on $\Ht_2$. For $g \in \G^+(\R)$, $z \in \Ht_2$, define $J(g,z)$ in the usual manner.

For $Z = \begin{pmatrix}\ast & \ast\\ \ast &z_{22}\end{pmatrix} \in \Ht_2$, we set $\widehat{Z}= \frac{i}{2}(\overline{Z}^T -Z)$ and $Z^\ast = z_{22}.$

Now, let us interpret the Eisenstein series of the last section as a function on $\Ht_2$. Recall the definitions of the global section $f_\Lambda(g,s) \in
\text{Ind}_{P(\A)}^{\G(\A)} (\Pi \times \delta_P^s) $, and the corresponding Whittaker function $W_{f_\Lambda} = \prod_v W_{f_\Lambda,v}$.

Also for $z\in \H_2$, put $$W'(z) = \det(g)^{-l/2}J(g,i)^lW_\Psi(g)$$ where $W_\Psi$ is the Whittaker function associated to $\Psi$ and $g\in GL_2^+(\R)$ is any element such that $g(i) =z$. Note that this definition does not depend on $g$.

\begin{lemma}Let $g_\infty \in \G^+(\R)$. Then
$$W_{f_\Lambda,\infty}(g_\infty, s) = \det(g_\infty)^{l/2}\det(J(g_\infty, i))^{-l}\left(\frac{\det(\widehat{g_\infty(i)}}
{\mathrm{Im}(g_\infty(i))^\ast}\right)
^{3(s+\frac{1}{2})-\frac{l}{2}}W'((g_\infty(i)^\ast).$$ Thus the function $$\det(g_\infty)^{-l/2}\det(J(g_\infty,i))^lW_{f_\Lambda,\infty}(g_\infty, s)$$ depends only on $g_\infty(i)$.
\end{lemma}
\begin{proof}
Let us write $$g_\infty = m^{(1)}(a)m^{(2)}(b) n k$$ where we use the notation of Subsection~\ref{s:eisenstein} with $a \in \R^\times$, $b = \begin{pmatrix}\alpha & \beta\\ \gamma & \delta \end{pmatrix} \in GL_2^+(\R)$, $n \in N(\R)$ and $k \in \widetilde{K}_\infty$.
Observe that $b(i)=(g_\infty(i))^\ast$. Then,~\eqref{defflambda}  tells us that  \begin{equation}W_{f_\Lambda}(g_\infty, s) =|a^2\mu_2(b)|^{3(s+\frac{1}{2}}\det(k)^{l/2}\det(b)^{l/2}J(b,i)^{-l}
\det(J(k, i))^{-l}W'((g_\infty(i)^\ast).\end{equation}

On the other hand, we can verify that \begin{equation}\det(\widehat{g_\infty(i)}) = \mu_2(b)^2|\det(J(g_\infty, i))|^{-2}.\end{equation} Also, \begin{equation}\det(J(g_\infty, i)) = a^{-1}\mu_2(b)(\gamma i + \delta)\det(J(k,i))\end{equation} and \begin{equation}\mathrm{Im}(g_\infty(i))^\ast = \mu_2(b)|\gamma i +\delta|^{-2}.\end{equation} Putting the above equations together, and using the fact that $|\det(J(k,i))|^{-2} =1$, we get the statement of the lemma.
\end{proof}

\begin{corollary}\label{eisensteinwelldefinedZ}Let $s \in \C$ be fixed. Then the function $$\det(g_\infty)^{-l/2}\det(J(g_\infty,i))^l
E_{\Psi, \Lambda}(g_\infty, s)$$ depends only on $g_\infty(i)$.
\end{corollary}
\begin{proof} Put $$r_\lambda = \begin{pmatrix}1&&&\\&\lambda&&\\&&\lambda&\\&&&1\end{pmatrix}.$$ The corollary follows immediately from the above lemma and the definition $$E_{\Psi, \Lambda}(g_\infty, s) = \sum_{\lambda \in \Q}\sum_{\gamma \in P(\Q) \bs \G(\Q)} W_{f_\Lambda,\infty}((r_\lambda)_\infty\gamma_\infty g_\infty, s) \left(\prod_{v<\infty}W_{f_\Lambda,v}((r_\lambda)_v,\gamma_v s)\right).$$

\end{proof}

Define the function $\E(Z,s)$ on $\Ht_2$ by \begin{equation}\E(Z,s)=\det(g_\infty)^{-l/2}\det(J(g_\infty, i))^lE_{\Psi, \Lambda}(g_\infty, \frac{s}{3} + \frac{l}{6}-\frac{1}{2}).\end{equation}

We know~\cite{fur} that the series defining $\E(Z,s)$ converges absolutely and uniformly for $s > 3 - \frac{l}{2}$. From now on, assume $l>6$. Then $\E(Z,0)$ is a holomorphic Eisenstein series on $\Ht_2.$ By~\cite{har} we know that $\E(Z,0)$ has algebraic Fourier coefficients.

Now, we consider the restriction of $\E(Z,0)$ to $\H_2$. Clearly, the resulting function also has algebraic Fourier coefficients.

Henceforth we abuse notation by using $\E(Z,0)$ to mean its restriction to $\H_2$.

\begin{proposition} Suppose $l>6$. Then $\E(Z,0)$ is a Siegel modular form of weight $l$ for $B(M) \cap U_2(N)$.
\end{proposition}
\begin{proof} By the above comments, $\E(Z,0)$ is holomorphic as a function on $\H_2$. Let $\gamma \in B(M) \cap U_2(N)$. We consider $\gamma$ as an element of $G(\Q)$ embedded diagonally in $G(\A)$. Write $\gamma = \gamma_\infty \gamma_f$ where $\gamma_f$ denotes the finite part. It suffices to show that $$\E(\gamma_\infty Z,0) = \det(J(\gamma_\infty, Z))^l\E(Z,0)$$ for $Z \in \H_2$.

 Let $g \in Sp(4, \R)$ be such that $g(i) = Z$; thus $\gamma_\infty g(i) = \gamma_\infty Z$.

We have \begin{align*}\E(\gamma_\infty Z,0)&=\det(g)^{-l/2}\det(J(\gamma_\infty , Z))^l\det(J(g , i))^lE_{ g, \Lambda}(\gamma g(\gamma_f)^{-1},  \frac{l}{6}-\frac{1}{2})\\&=\det(J(\gamma_\infty , Z))^l(\det(g)^{-l/2}\det(J(g , i))^lE_{ g, \Lambda}( g,  \frac{l}{6}-\frac{1}{2}))\\&=\det(J(\gamma_\infty, Z))^l\E(Z,0)\end{align*}

\end{proof}

For any congruence subgroup $\Gamma$ of $Sp(4, \Z)$ let $V(\Gamma)$ denote the quantity $[Sp(4, \Z):\Gamma]^{-1}$.

Suppose $f(Z)$ and $g(Z)$ are Siegel modular forms of weight $l$ for some congruence subgroup. We define the Petersson inner product $$\langle f, g \rangle = \frac{1}{2}V(\Gamma) \int_{\Gamma \bs \H_2}f(Z)\overline{g(Z)}(\det(Y))^{l-3}dX dY$$ where $Z= X+iY$ and $\Gamma$ is any congruence subgroup such that $f, g$ are both Siegel modular forms for $\Gamma$. Note that this definition does not depend on the choice of $\Gamma$.

Also for brevity, we put $\Gamma_{M,N}=B(M) \cap U_2(N)$ and $V_{M,N} =V(\Gamma_{M,N})$.

\begin{proposition} Assume $l>6$. Define the global integral $Z(s)$ as in~\eqref{ss:global}. Then $$Z(\frac{l}{2}-\frac{1}{2}) = \langle\E(Z,0), F \rangle.$$

\end{proposition}

\begin{proof}
By definition, we have $$Z(\frac{l}{2}-\frac{1}{2}) =\int_{Z_G(\A)G(\Q)\backslash G(\A)}E_{\Psi,\Lambda}(g,0)\overline{\Phi}(g)dg.$$

It suffices to prove that \begin{equation}\label{e:clasint}\int_{Z_G(\A)G(\Q)\backslash G(\A)}E_{\Psi,\Lambda}(g,s)\overline{\Phi}(g)dg= \frac{V_{M,N} }{2} \int_{\Gamma_{M,N} \bs \H_2}\E(Z,0)\overline{F(Z)}\det(Y)^{l-3}dX dY.\end{equation}

Recall the definition of the compact open subgroup $K^{\G}$ from Subsection~\ref{s:eisensteindef}. The integrand on the left side is right invariant under $K^G=(K^{\G}K_\infty) \cap G(\A)$. Furthermore vol$(K^G) = V_{M,N}$ and we have $$Z_G(\A)G(\Q)\backslash G(\A)/K^G = \Gamma_{M,N} \bs \H_2.$$ Now~\eqref{e:clasint} follows from the above comments and the observation that the $G(\R)^+$-invariant measure on $\H_2$ and $dg$
are related by $dg =  \frac{1}{2}(\det(Y))^{-3}dX dY.$
\end{proof}

For $\sigma \in$ Aut$(\C)$, and an arbitrary Siegel modular form $\Theta$, denote by $\Theta^\sigma$ (resp $\Theta^{-}$) the Siegel modular form obtained by applying $\sigma$ (resp. complex conjugation) to all the Fourier coefficients of $\Theta$.

\begin{theorem}\label{t:specialvalues}Let $F,  g$ be as defined in Subsection~\ref{s:newformdef} with $l>6$. Then, for $\sigma \in Aut(\C/\overline{\Q})$, we  have
$$\left(\frac{L(\frac{l}{2}-1, F \times  g)}{\pi^{5l-8}\langle F, F^{-} \rangle \langle  g,  g \rangle} \right)^\sigma = \frac{L(\frac{l}{2}-1, F \times  g)}{\pi^{5l-8}\langle F^\sigma, (F^\sigma)^{-} \rangle \langle  g,  g \rangle}.$$
\end{theorem}
\begin{proof} By~\cite{har} we know that the orthogonal complement of the space of Siegel cusp forms of any level is spanned by Eisenstein series which have algebraic Fourier coefficients. It follows from the theorem at the top of p. 460 in~\cite{gar2} we have $$\left(\frac{\langle\E(Z,0), F^{-} \rangle}{\langle F, F^{-} \rangle}\right)^\sigma = \frac{\langle\E(Z,0)^\sigma, (F^\sigma)^{-} \rangle}{\langle F^\sigma, (F^\sigma)^{-} \rangle}$$

Now, by~\cite{har} we know that $\E(Z,0)^\sigma = \E(Z,0)$. Also, since all the Hecke eigenvalues of $F$ are totally real and algebraic, we have $$L(F \times  g) = L(F^\sigma \times  g) = L(F^{-} \times  g).$$ Therefore, from the above proposition and the remark at the end of Theorem~\ref{t:globalmain}, it follows that \begin{equation}\label{specialeq}\left(\frac{\pi^{4-2l}\overline{a(F^{-}, \Lambda)}L(\frac{l}{2}-1, F \times g)}{\zeta(l-2)L(\frac{l-1}{2},\sigma \times \rho(\Lambda))\langle  F, F^{-} \rangle } \right)^\sigma = \frac{\pi^{4-2l}\overline{a((F^\sigma)^{-}, \Lambda)}L(\frac{l}{2}-1, F \times  g)}{\zeta(l-2)L(\frac{l-1}{2},\sigma \times \rho(\Lambda))\langle F^\sigma, (F^\sigma)^{-} \rangle }.\end{equation}

It is well-known that $\zeta(l-2)\pi^{2-l} \in \Q.$ Also using the same argument as in the proof of~\cite[Theorem 4.8.3]{fur}, we have $$\frac{L(\frac{l-1}{2},\sigma \times \rho(\Lambda)}{\pi^{2l-2}\langle g,  g \rangle} \in \overline{\Q}.$$ These facts, when substituted in~\eqref{specialeq} give the assertion of the theorem.

\end{proof}

The above theorem implies the following corollary.

\begin{corollary}Let $F,  g$ be as defined in Subsection~\ref{s:newformdef} with $l>6$ and furthermore assume that $F$ has totally real algebraic Fourier coefficients. Then
$$\frac{L(\frac{l}{2}-1, F \times  g)}{\pi^{5l-8}\langle F, F \rangle \langle  g,  g \rangle} \in \overline{\Q}.$$
\end{corollary}

\textbf{Remark.} Newforms for $GL(2)$, when normalized, automatically have algebraic Fourier coefficients. A similar statement is not known for Siegel newforms (among other things, we do not know multiplicity one for $GSp(4)$). However by~\cite{gar2} we do know the following: The space of Siegel cusp forms for a principal congruence subgroup has a \emph{basis}  of Hecke eigenforms with totally real algebraic Fourier coefficients.

\section{Further questions}

It is of interest to investigate the special values of $L(s, F \times g)$ more closely. In particular, we may ask the following questions.

\begin{enumerate}
\item Does the expected reciprocity law hold for the special value $L(\frac{l}{2}-1, F \times  g)$, i.e., does Theorem~\ref{t:specialvalues} hold for $\sigma$ any automorphism of $\C$?

\item Do we have similar special value results for the other `critical' values of $L(s, F \times  g)$ as predicted by Deligne's conjectures?
    \end{enumerate}

  We can answer the first question if we know precisely the behavior of the Fourier coefficients of $\E(Z,0)$ under an automorphism of $\C$. For the second, we would like to know similar facts for $\E(Z,s)$ with $s$ lying outside the range of absolute convergence of the Eisenstein series. It seems hard to extract these directly, as our Eisenstein series --- being induced from an automorphic
representation of $GL(2)$ sitting inside the Klingen parabolic --- is rather complicated.

However, using a `pullback formula', we can switch to a more
standard Siegel-type Eisenstein series on a higher rank group. More precisely, we will derive, in a sequel to this paper~\cite{sah2}, a variant of our integral representation which involves pulling back an Eisenstein series from $GU(3,3)$. Incidentally, this second integral
representation looks similar to the Garrett--Piatetski-Shapiro--Rallis integral representation for the
triple product $L$-function.

Let us describe this second integral representation in more detail.

Let
$\widetilde{G}^{(3)}=GU(3,3;L),  \widetilde{F} = GU(1,1;L)$. Let $H_1$ denote the subgroup
of $G \times \widetilde{F}$ consisting of elements $h=(h_1,h_2)$ such that $h_1\in G, h_2\in  \widetilde{F}$ and
$\mu_2(h_1) =\mu_1(h_2)$. We fix a certain embedding $H_1 \hookrightarrow \widetilde{G}^{(3)}$. Let $P_{\widetilde{G}^{(3)}}$ be the Siegel
parabolic of $\widetilde{G}^{(3)}$. Given a section $\Upsilon(s)$ of
$\text{Ind}_{P_{\widetilde{G}^{(3)}}}^{\widetilde{G}^{(3)}}(\Lambda \times |\, |^{3s})$ define the Eisenstein series
$E_{\Upsilon}(h,s)$ on $\widetilde{G}^{(3)}(\A)$ in the usual manner.

Now consider the global integral
$$Z(s)= \int_{Z_{\widetilde{G}^{(3)}}(\A)H_1(\Q)\backslash H_1(\A)} \Lambda^{-1}(\det h_2)\overline{\Phi}(h_1)\Psi(h_2)E_{\Upsilon}(h_1,h_2,s)dh$$ where $h=(h_1,h_2)$.
We will prove in~\cite{sah2} that for a suitable $\Upsilon$,  $$Z(s) = L\left(3s + \frac{1}{2},F \times
 g\right) \times \text{(normalizing factor)}.$$

So, to answer the questions stated in the beginning of this section it suffices to study the (simpler) Eisenstein series $E_{\Upsilon}(h,s)$. Indeed, the action of Aut$(\C)$ on the Fourier coefficients is then known, enabling us to answer the first question. For the second there seem to be two possible strategies: the theory of nearly holomorphic functions due to Shimura~\cite{shibook2}, or a Siegel-Weil formula based attack explained by Harris in his papers~\cite{har97,har05}.

In~\cite{sah2}, the approach sketched in this section will be fleshed out and the special value properties of the $L$-function investigated in more detail.

\subsection*{Acknowledgements}
The author would like to thank M. Harris for some helpful suggestions (whose importance will be more apparent in the sequel to this paper), D. Lanphier for useful discussions and T. Tsankov for proof-reading a part of this paper.

The author thankfully acknowledges his use of the software MAPLE for performing many of the computations for this paper.

This work was done while the author was a graduate student at Caltech and represents part of his Ph.D. dissertation. The author thanks his advisor Dinakar Ramakrishnan for guidance, support and many helpful discussions.

\bibliography{lfunction}

\end{document}